\documentclass[12pt]{amsart}
\usepackage{amsmath,amssymb,amsfonts,latexsym}
\usepackage{mathrsfs}
\usepackage[dvips]{geometry}
\usepackage{array}
\usepackage{color}

\usepackage{epsfig, graphicx} 
\usepackage{multirow}

\begin{document}
\title[Slices for maximal parabolic subalgebras]{Slices for maximal parabolic subalgebras of a semisimple Lie algebra $^*$}
\thanks{$^*$ This work was supported in part by the Israel Scientific Foundation grant 797/14 and in part by LABEX MILYON (ANR-10-LABX-0070)
of Universit\'e de Lyon, within the program ``Investissements d'Avenir'' (ANR-11-IDEX-0007)
operated by the French National Research Agency (ANR)}
\author[Florence Fauquant-Millet and Polyxeni Lamprou ]{Florence Fauquant-Millet and Polyxeni Lamprou }

\maketitle


\newenvironment{theorem}{\noindent \bf Theorem.\it}{ \rm }
\newenvironment{proposition}{\noindent \bf Proposition.\it}{ \rm }
\newenvironment{lemma}{\noindent \bf Lemma.\it}{ \rm }
\newenvironment{corollary}{\noindent \bf Corollary.\it}{ \rm }
\newenvironment{definition}{\noindent \bf Definition.\it}{ \rm }
\newenvironment{remark}{\noindent \bf Remark.\it}{ \rm }

\newcommand{{\ind}}{\operatorname{ind}}
\newcommand{{\ad}}{\operatorname{ad}}
\newcommand{{\card}}{\operatorname{Card}}
\newcommand{{\rk}}{\operatorname{rk}}
\newcommand{{\id}}{\operatorname{id}}
\newcommand{{\trp}}{\mathfrak{p}_{\pi', E}}
\newcommand{{\liea}}{\mathfrak{a}}
\newcommand{{\lieg}}{\mathfrak{g}}
\newcommand{{\liep}}{\mathfrak{p}}
\newcommand{{\lieh}}{\mathfrak{h}}
\newcommand{{\lien}}{\mathfrak{n}}
\newcommand{{\liel}}{\mathfrak{l}}
\newcommand{{\lieb}}{\mathfrak{b}}
\newcommand{{\liesl}}{\mathfrak{sl}}
\newcommand{{\bbN}}{\mathbb{N}}
\newcommand{{\bbC}}{\mathbb{C}}

\def\a{\mathfrak a}
\def\g{\mathfrak g}
\def\h{\mathfrak h}
\def\p{\mathfrak p}
\def\n{\mathfrak n}
\def\b{\mathfrak b}
\def\l{\mathfrak l}
\def\o{\mathfrak o}
\def\ep{\varepsilon}
\def\al{\alpha}
\def\s{\mathfrak s}
\def\l{\mathfrak l}

\begin{abstract}
Let $\p$ denote a maximal (truncated) parabolic subalgebra of a simple Lie algebra $\g$. In [F. Fauquant-Millet and A. Joseph, Semi-centre de l'alg\`ebre enveloppante d'une sous-alg\`ebre parabolique d'une
alg\`ebre de Lie semi-simple, Ann. Sci. \'Ecole. Norm. Sup. (4)
{\bf{38}} (2005) no.2, 155-191] it was shown in many cases that the Poisson centre $Y(\p)$ is a polynomial algebra. We construct a slice for the coadjoint action of $\p$, thus extending a theorem of Kostant. The role of the principal $\s\l_2$-triple is played by an adapted pair, a notion introduced in [A. Joseph and P. Lamprou, Maximal Poisson commutative subalgebras for truncated parabolic subalgebras of maximal
index in $\mathfrak{sl}_n$, Transform. Groups {\bf 12} (2007), no. 3, 549-571].
\end{abstract}
\noindent Key words : Slices, Truncated Parabolic Subalgebras, Poisson Centre.\\
AMS Classification : 17B35, 16W22.

\section{Introduction}
Throughout this paper, we are working over an algebraically closed field $k$ of characteristic $0$.

\subsection{}Let $\mathfrak{a}$ be a finite dimensional algebraic Lie algebra, $\mathfrak{a}^*$ its dual space and $S(\mathfrak{a})$ the symmetric algebra of $\mathfrak{a}$. Then if $\{x_1,\, x_2,\, \dots, \,x_n\}$ is a basis of
$\mathfrak{a}$, $S(\liea)$ identifies with $k[x_1,\, x_2,\, \dots,\, x_n]$, the polynomial algebra with generators $x_1,\,\ldots,\,x_n$, and with the
algebra of polynomials on $\mathfrak{a}^*$. For any two polynomials $f, g$
in $S(\mathfrak{a})$ we define the Poisson bracket of $f$ and
$g$ by the formula :
$$\{f,\, g\} = \sum\limits_{i=1}^n\sum\limits_{j=1}^n\frac{\partial f}{\partial
x_i}\frac{\partial g}{\partial x_j}[x_i,\, x_j],$$
where $[\,,\,]$ is the Lie bracket of $\a$.
For any two
elements $x, y \in \mathfrak{a}$ we have that $\{x,\, y\} = [x,\, y]$
and each $g \mapsto \{g,\, f\}$ is a derivation of the associative
algebra $S(\mathfrak{a})$ for fixed $f\in S(\mathfrak{a})$.

We denote by $\ad$ the adjoint action of $\a$ on itself defined by its Lie bracket.

\subsection{} 
Denote by $Y(\mathfrak{a})$ the Poisson centre of $S(\mathfrak a)$, that is the space :
$$Y(\mathfrak{a}) = \{f\in S(\mathfrak{a})\mid \forall\,g\in S(\mathfrak{a}),\;\; \{f,\, g\} = 0\}.$$
Note that $Y(\mathfrak{a})=S(\a)^{\a}:=\{f\in S(\a)\mid \forall\, x\in \a,\;\; \{x,\,f\}=0\}$ and that $Y(\a)$ is a subalgebra
of  $S(\a)$.
We denote by $Sy(\a)$ the Poisson semicentre of $S(\a)$, $Sy(\a):=\{f\in S(\a)\mid \forall\, x\in \a,\;\;\{x,\,f\}\in k\,f\}$; it is also
a subalgebra of $S(\a)$.
Clearly, $Y(\a)\subset Sy(\a)$ and this inclusion is in general strict. We will say that $S(\a)$ has no proper semi-invariants when
equality holds. It is the case for instance  when $\a$ is a truncated (bi)parabolic subalgebra of a semisimple Lie algebra, \cite[lemme 2.5]{FJ3}.

The Poisson centre $Y(\liea)$ - and in particular whether or not it is a polynomial algebra - has been studied in several cases, for example in the case where $\liea$ is semisimple, or a centralizer of a nilpotent element of a semisimple Lie algebra \cite{PPY}, or a (truncated) (bi)parabolic subalgebra of a semisimple Lie algebra \cite{FJ2, FJ3, J6, J7}.

\subsection{}

Recall that $\a$ acts via the coadjoint action on $\a^*$ and denote  this action still by $\ad$. Set $\ind \, \mathfrak{a} =\operatorname{min}_{\xi\in a^*}\{\operatorname{codim}\,(\ad\mathfrak{a})\,\xi\}$ and call an element $\xi$ in $\mathfrak{a}^*$ regular if $\operatorname{codim}\,(\ad\mathfrak{a})\,\xi= \ind \,\mathfrak{a}$.
Denote by $\a^*_{reg}$ the set of regular elements of $\a^*$.

\subsection{}\label{ap} Assume that $Y(\mathfrak{a})$ is a polynomial algebra  in $\ell := \ind \mathfrak{a}$ generators. Let $\{f_1,\, f_2,\, \dots,\, f_{\ell}\}$ be
a set of homogeneous generators in $Y(\mathfrak{a})$ and let $\deg
f_i$ denote the degree of $f_i$ for all $i,\,1\leq i\leq
\ell$.

In \cite{JL} an {\it adapted pair} for $\liea$ was defined as a pair $(h,\,y)\in \liea\times \liea^*$ such that :
\begin{enumerate}
\item[(i)] $h$ is an $\operatorname{ad}$-semisimple element, $y$ is regular and
$(\operatorname{ad}h)\,y=-y$.

\item[(ii)] Let $V$ be an $\ad h$-stable subspace of $\mathfrak{a}^*$ such
that $(\operatorname{ad}\mathfrak{a})\,y\oplus V = \mathfrak{a}^*$. The $\ad h$-eigenvalues $m_i,\,1\leq i\leq
\ell$, of an $\operatorname{ad}h$-stable basis of $V$ are all non-negative and satisfy :
$\sum\limits_{i=1}^{\ind \mathfrak{a}}\deg f_i
= \sum\limits_{i=1}^{\ind \mathfrak{a}}(m_i + 1).$
\end{enumerate}

\subsection{} Recall the Slice Theorem~\cite[Thm 6]{JL}, which
obtains from the analysis of Kostant \cite[Thm 5, Thm 7]{K}:\\

\begin{theorem}\label{slice} Assume that $Y(\liea)$ is polynomial in $\ind\, \liea$ generators  and let $(h,\, y)$ be an adapted pair for $\liea$
in the sense of~\ref{ap}. Then :
\begin{enumerate}
\item The map $Y(\mathfrak{a})\longrightarrow R[y + V]$ defined by
restriction of functions is an isomorphism of algebras, where $R[y+V]$ denotes the algebra of regular functions on $y+V$.
\item One has $m_i + 1 = \deg f_i$, for all $i,\, 1\leq i\leq
\ind \mathfrak{a}$, up to a permutation of indices.
\end{enumerate}
\end{theorem}

A linear subvariety $y+V$ satisfying (1) of the above theorem is called a {\it Weierstrass section} ~\cite{FJ4} or an {\it algebraic slice}~\cite[7.6]{J8}.
When $\mathfrak{a}=\lieg$ is a finite dimensional semisimple Lie algebra, it is well-known that the Poisson centre $Y(\g)$ is polynomial in $\ind \,\mathfrak{g} =\operatorname{rk}\,\mathfrak{g}$ generators (for an exposition see {\cite[7.3.8]{D}}). The above theorem is known as the Kostant Slice Theorem. In this case $V=\mathfrak g^x$, the centralizer of $x$ in $\lieg$, where $\{x,\, h, \,y\}$ is a principal $\mathfrak{sl}_2$ - triple with $[h,\, x] = x,\,[h,\, y] = -y$ and $[x,\, y]=h$.

Let $G$ be the adjoint group of $\g$. The affine space $y+\lieg^x$ is called  a {\it slice for the coadjoint action of $\lieg$} \cite[7.3]{J8} since every $G$ - orbit in $G(y+\lieg^x)$ meets  transversally $y+\lieg^x$ at exactly one point and $\overline{G(y+\lieg^x)} = \lieg^*$ where
$\overline{G(y+\lieg^x)}$ is the Zariski closure of $G(y+\lieg^x)$.  Moreover in that particular case, $G(y+\lieg^x)=\lieg^*_{reg}$. (Note that in general, $A(y+V)$ is a proper subset of $\liea^*_{reg}$, where $A$ is the adjoint group of $\a$).

\subsection{}\label{adaptedpair1} In \cite{JS}, under further assumptions on $\liea$, namely that $\liea$ is unimodular and that its fundamental semi-invariant is an invariant, the authors showed that condition (i) of the definition of an adapted pair implies condition (ii) and actually (2) of the Theorem \ref{slice}. Furthermore,
as a consequence of \cite[Thm. 1.11 (i)]{D3}, any Lie algebra $\a$ such that $S(\a)$ has no proper semi-invariants is unimodular.
Thus a truncated parabolic subalgebra $\a$ of a semisimple Lie algebra $\lieg$ is unimodular  and  its fundamental semi-invariant is an invariant.

Since in this work we are interested precisely in  truncated parabolic algebras, we will reformulate the definition of an adapted pair and the Slice Theorem, following \cite{JS}.\\

\begin{definition} An adapted pair for $\liea$ is a pair $(h, \,y)\in \liea\times \liea^*$ such that $h$ is an $\operatorname{ad}$ - semisimple element, $y$ is regular and
$(\operatorname{ad}h)\,y=-y$.
\end{definition}\\

\noindent Then Theorem \ref{slice} becomes :\\

\begin{theorem} \cite[Corollary 2.3]{JS}
Let $\liea$ be a finite dimensional unimodular Lie algebra whose fundamental semi-invariant is an invariant (for example $\liea$ is a truncated parabolic) and suppose that $\liea$ admits an adapted pair $(h, \,y)$. Let $V$ be a subspace of $\mathfrak{a}^*$ such
that
$(\operatorname{ad}\mathfrak{a})\,y\oplus V = \mathfrak{a}^*$
and let $m_i,\,1\leq i\leq
\ell:=\ind \mathfrak{a}$ be the $\ad h$-eigenvalues of an $\operatorname{ad}h$-stable
 basis of $V$. If $Y(\liea)$ is polynomial in $\ell$ generators, then
\begin{enumerate}
\item The map $Y(\mathfrak{a})\longrightarrow R[y + V]$ defined by
restriction of functions is an isomorphism of algebras.
\item The degrees of a set of homogeneous generators of $Y(\liea)$ are the $m_i + 1,\, 1\leq i\leq
\ell$.
\end{enumerate}
\end{theorem}

If moreover $S(\a)$ has no proper semi-invariants then, after \cite[Lemma 3.2]{FJ4} the Weierstrass section $y+V$ is a slice for the coadjoint action of $\a$, in the sense of \ref{slice}.

\subsection{}\label{theoryap} Suppose that we have constructed an adapted pair $(h, \,y)$ for $\mathfrak a$  and retain the notations and hypotheses of Theorem \ref{adaptedpair1}. By the above, if $Y(\liea)$ is polynomial, the degrees of a set of generators are $\{m_i+1\,|\, 1\le i\le \ell\}$. This fact gives us indications for the polynomiality or non-polynomiality of $Y(\mathfrak a)$. If for example $m_i<0$ for some $i$, or if we can choose different complements $V,\, V'$ on which $\ad\, h$ has different set of eigenvalues we conclude that $Y(\liea)$ is not polynomial. For the truncated parabolic subalgebras we know by \cite{FJ2} that $Y(\liea)$ is included, up to a double gradation which respects the degrees, in a polynomial algebra whose degrees of generators are known, hence we have some information about the possible degrees of invariants. If this information is contradicted by the set $\{m_i\,|\, 1\le i\le \ell\}$ again we may conclude that $Y(\liea)$ is not polynomial. In any case, the existence of an adapted pair restricts us to searching for invariants having a defined set of degrees.

\subsection{} When $\liea$ is a truncated parabolic subalgebra of a semisimple Lie algebra $\lieg$, we know in many cases that $Y(\liea)$ is polynomial; for example, by \cite{FJ2} this is always true when $\lieg$ has only components of type $A$ or $C$ and by \cite{PPY}, this is also the case when the truncated parabolic is the centralizer of the highest root vector and $\lieg$ is simple not of type $E_8$. We will construct an adapted pair for all maximal parabolic subalgebras where we know (by the criterion of \cite{FJ2}, also given in Theorem~\ref{borel})
that the Poisson centre is polynomial. (This was done already in \cite{J6bis}, in the case of the centralizer.)\\ 

\noindent {\bf Acknowledgements.} We thank A. Joseph for suggesting this problem to us and in general for introducing us to truncated parabolic subalgebras and adapted pairs many years ago.

\section{The truncated parabolic subalgebra and the centre of its symmetric algebra}

\subsection{Truncated parabolic subalgebras of $\lieg$}

\subsubsection{}\label{standardnotation} Let $\mathfrak{g}$ be a finite dimensional
semisimple Lie algebra, $\mathfrak{h}$ a fixed
Cartan subalgebra of $\mathfrak{g}$, $\Delta$ the root system of
$\mathfrak{g}$ with respect to $\mathfrak{h}$ and $\pi$ a chosen set of
simple roots. We will adopt the labeling of \cite[Planches I-IX]{BOU} for the simple roots in $\pi$. Let $\Delta^+$ (resp. $\Delta^-$) denote the set of positive (resp. negative) roots. One has that $\Delta^-=-\Delta^+$ and $\Delta=\Delta^+\sqcup \Delta^-$. For any $\alpha\in \Delta$, let $\mathfrak g_\alpha$ denote the corresponding root space of $\mathfrak g$. Then $\mathfrak{g}
= \mathfrak{n}\oplus \mathfrak{h}\oplus \mathfrak{n}^-$, where
$\mathfrak{n}=\bigoplus\limits_{\alpha\in \Delta^+}\mathfrak g_\alpha$ and $\mathfrak{n}^-=\bigoplus\limits_{\alpha\in \Delta^-}\mathfrak g_\alpha$.
For all $\alpha\in\pi$, denote by $\alpha^\vee$ the corresponding coroot. One has that $\h=\bigoplus\limits_{\alpha\in\pi} k\,\alpha^{\vee}$.

\subsubsection{} For any subset $\pi^{\prime}$ of $\pi$, we will denote by
$\Delta_{\pi'}$ the set of roots generated by $\pi'$. We also denote by $\Delta_{\pi'}^+, \, \Delta_{\pi'}^-$ the sets of positive and negative roots in $\Delta_{\pi'}$ respectively. One defines the standard parabolic
subalgebra $\mathfrak{p}_{\pi^{\prime}}$ to be the algebra $\mathfrak{p}_{\pi'} = \lien \oplus
\lieh \oplus \lien_{\pi'}^-$ where $\lien^-_{\pi'}=\bigoplus\limits_{\alpha\in \Delta^-_{\pi'}}\g_\alpha$. Its opposed algebra then is $\mathfrak{p}_{\pi'}^- = \lien^- \oplus
\lieh \oplus \lien_{\pi'}$, with $\lien_{\pi'}$ defined similarly. The dual algebra $\mathfrak p_{\pi'}^*$ identifies with $\liep_{\pi'}^-$ via the Killing isomorphism $\varphi$.

\subsubsection{} For any parabolic subalgebra $\liep_{\pi'}$ of a simple Lie algebra $\lieg$ such that $\pi'\subsetneq\pi$, the Poisson centre of $S(\mathfrak p_{\pi'})$ reduces to scalars \cite[7.9]{J6}. In \cite{FJ2} and \cite{FJ3}, the authors study the Poisson semicentre $Sy(\liep_{\pi'})$ of $S(\liep_{\pi'})$. Since $\liep_{\pi'}$ is algebraic there exists a
canonical truncation $\liep_{\pi', E}$ of $\mathfrak p_{\pi'}$ such that $Sy(\liep_{\pi'}) =
Y(\liep_{\pi', E})=Sy(\liep_{\pi', E})$ \cite[2.4, 2.5, B.2]{FJ3}. In \cite[Section 5]{FJ3} $\liep_{\pi', E}$
has been described explicitly. It has the form
$\liep_{\pi', E} = \lien\oplus \lieh_E\oplus \lien^-_{\pi'}$,
where $\lieh_E$ is a truncation of $\lieh$; it is called the truncated Cartan subalgebra.
Its description involves some combinatorics of the Dynkin diagram of $\lieg$,
which we describe below.

\subsubsection{}\label{Orbits} Define an involution $j$ on $\pi$ by $j =
-w_0$, where $w_0$ denotes the longest element in the Weyl group
of $\Delta$. Extend the involution $i = -w_0^{\prime}$ of
$\pi^{\prime}$, where $w_0^{\prime}$ denotes the longest element
in the Weyl group of $\Delta_{\pi'}$, on $\pi\setminus
\pi^{\prime}$ as follows. For all $\alpha\in \pi \setminus
\pi^{\prime}$, if $j(\alpha)\in\pi \setminus \pi^{\prime}$ set
$i(\alpha) = j(\alpha)$. If now $j(\alpha)\in \pi^{\prime}$ let
$r$ be the smallest integer such that $j(ij)^r(\alpha) \notin
\pi^{\prime}$ and set $i(\alpha) = j(ij)^r(\alpha)$. Let $\langle ij\rangle$
denote the group generated by the element $ij$ and $E$ the set of
$\langle ij\rangle$ -  orbits in $\pi$. Set $E_1 := \{\Gamma\in E\mid j\Gamma \neq
\Gamma\}$ and $E_2 : = \{\Gamma\in E\mid j\Gamma = \Gamma\}$. One
has $E = E_1\bigsqcup E_2$. Notice that an orbit $\Gamma$ in $E$
is also an $\langle i,\, j\rangle$ (the group generated by $i$ and $j$) - orbit if
and only if $\Gamma\in E_2$. The index of $\liep_{\pi', E}$ is equal to the number of the $\langle ij\rangle$ - orbits in $\pi$, that is $\ind \trp = \card E$ \cite[7.14]{J6}.

\subsubsection{}\label{truncatedCartan} Denote by $\ell$ the standard length function on the
Coxeter group $\langle i,\, j\rangle$ and let
$\{\varpi_{\alpha}\}_{\alpha\in\pi}$ be the set of fundamental weights
of $\mathfrak{g}$ (sometimes we write $\{\varpi_i\}_{\alpha_i\in \pi}$). For all $\Gamma\in E$ which does not lie
entirely in $\pi^{\prime}$, define $h_{\Gamma} =
\sum\limits_{s\in\langle i,\,j\rangle}
(-1)^{\ell(s)}\varphi^{-1}(\varpi_{s\alpha})$ for a choice of
$\alpha\in\Gamma$. For simplicity, we will write $h_{\Gamma} =
\sum\limits_{s\in\langle i,\, j\rangle}(-1)^{\ell(s)}\varpi_{s\alpha}$. Denote
by $E^{\prime}_1$ the subset of $E_1$ of orbits which meet
$\pi\setminus \pi^{\prime}$. Define
$\mathfrak{p}_{\pi^{\prime}}^{\prime}$ to be the derived
subalgebra of $\mathfrak{p}_{\pi^{\prime}}$,
$\mathfrak{p}_{\pi^{\prime}}^{\prime} =
[\mathfrak{p}_{\pi^{\prime}}, \mathfrak{p}_{\pi^{\prime}}]$. Set
$\mathfrak{h}^{\prime} =
\mathfrak{p}^{\prime}_{\pi^{\prime}}\bigcap \mathfrak{h}$ and
$\mathfrak{h}_E = \mathfrak{h}^{\prime} + \sum\limits_{\Gamma\in
E^{\prime}_1} k\,h_{\Gamma}$. It turns out that the latter
sum is direct. Then $\mathfrak{p}_{\pi^{\prime}, E}$ is the
algebra $\mathfrak{p}_{\pi^{\prime}, E} = \mathfrak{n} \oplus
\mathfrak{h}_E\oplus \mathfrak{n}^-_{\pi'}$.

\subsubsection{}\label{wpi =id}
When $j=\id$, which is true in all cases outside type $A_n,\,D_{2n+1}$ and
$E_6$, one has that $\langle ij\rangle =
\langle i, j\rangle =\langle i\rangle$, $E_1 =\emptyset$ and $E = E_2$. The involution $i$
extends on $\pi\setminus \pi'$ by the identity. It then follows that for $\Gamma\in E$ either $\Gamma\subset \pi\setminus
\pi'$ or $\Gamma\subset \pi'$. In the former case, $\Gamma$ is a singleton.
Moreover
$\mathfrak{h}_E = \mathfrak{h}^{\prime}$.

\subsubsection{}\label{hEmax}
For a truncated maximal parabolic $\liep_{\pi', E}$ associated to $\pi'=\pi\setminus\{\alpha_s\}$, one has always $i(\alpha_s)=\alpha_s$, even when
$j\neq \id$. Thus $E_1'=\emptyset$ and then $\h_E=\h'$.

\subsection{The Poisson centre of $\trp$}\label{bounds}
As we already said, the algebra $Sy(\liep_{\pi'})=Y(\trp)$ has been studied in \cite{FJ2}, \cite{FJ3}. More precisely, there have been found certain upper and lower bounds for $Sy(\liep_{\pi'})$; when these bounds coincide, the algebra $Sy(\liep_{\pi'})$ is polynomial in $\ind\, \liep_{\pi', E}$ generators. In this case, the weights and the degrees of a set of homogeneous generators have been computed explicitly. In this section, we will briefly describe the criterion of \cite{FJ2}, which implies the polynomiality of $Y(\trp)$. We first need some further background, namely the notion of the Kostant cascade.

\subsubsection{}\label{cascade}(See also \cite[2.2]{J1}).
Let $\mathfrak{g}$ be as in section \ref{standardnotation} and set
$\mathcal{D} = \{\lambda\in\mathfrak{h}^*_{\mathbb{R}}\mid
(\lambda, \alpha) \geq 0,\, \forall \alpha\in\pi\}$
where $\h^*_{\mathbb R}$ is the real subspace of $\h^*$ generated by
$\pi$. Given
$\lambda\in\mathcal{D}$, define $\Delta_{\lambda}^{\pm} =
\{\alpha\in \Delta^{\pm}\mid (\lambda,\, \alpha)  = 0\},\,
\Delta_{\lambda} = \Delta_{\lambda}^+\cup \Delta_{\lambda}^-$.
 Let $\Delta = \bigsqcup\limits_i
\Delta_{i}$ be a decomposition of $\Delta$ to irreducible root
systems. Let $\beta_i$ be the unique highest root in $\Delta_i$. Then
$\beta_i\in \mathcal{D}$ for all $i$, $(\Delta_i)_{\beta_i}$
is a root system and it decomposes into a union of irreducible root
systems $(\Delta_i)_{\beta_i} = \bigsqcup\limits_j \Delta_{ij}$ with $\beta_{ij}$ the highest root of $\Delta_{ij}$.
This procedure defines a subset $\mathcal{K}$ of
$\bbN^*\cup{\bbN^*}^2\cup\cdots$ and a maximal set $\beta_{\pi} =
\{\beta_K\}_{K\in \mathcal{K}}$ of strongly orthogonal positive roots. The
set $\mathcal{K}$ can be endowed with a partial order $\leq$ as
follows : for all $K, \, L\in \mathcal{K}$ we say that $K\leq L$
if and only if $L=K$ or $L=\{K, n_1, n_2, \dots, n_s\}$, with $n_i\in
\bbN^*$ for all $i,\, 1\leq i\leq s$ and $s\geq 1$. The set
$\mathcal{K}$ - called the Kostant cascade -  is
described in \cite[Table III]{J1} for every simple Lie algebra $\mathfrak{g}$.

The set $\beta_{\pi}$ of strongly orthogonal positive roots for a simple Lie algebra $\g$, which may be also found in \cite[Tables I and II]{J1}, is given for the reader's convenience in Table I where the indexation of the simple roots is as in \cite[Planches I-IX]{BOU}.

If the initial root system $\Delta$ is irreducible of type $A,\,C,\,E_6,\,F_4,\,G_2$, the subsystems are
also irreducible and hence the Kostant cascade $\mathcal K$ and the set $\beta_{\pi}$ are totally ordered.
Thus in these cases, the elements $\beta_K$, $K\in\mathcal K$, may be simply indexed by $\mathbb N$, and are denoted by
$\beta_i$, $1\le i\le \rm{card}({\mathcal K})$. We understand $\beta_i<\beta_j$ if and only if $i<j$.

In other types, the order $\le $ on $\mathcal K$ is not a total order.  For the elements $\beta_K$, $K\in\mathcal K$,
we use the notations $\beta_i$, $\beta_{i'}$,  for type $B_n$ or $D_{2n+1}$, or $\beta_i$, $\beta_{i'}$ and $\beta_{i''}$
for type $D_{2n}$, $E_7$ or $E_8$ with order relation $\beta_i<\beta_j$ if and only if $i<j$, and $\beta_i<\beta_{i'},\, \beta_{i''}$.



For any rational number $r\in \mathbb{Q}$ denote by $[r]$ the integer such that, $r-1 < [r]\leq r$.
\newpage
\begin{center}
\begin{tabular}{|l|l|}\hline
Type&The set $\beta_{\pi}$ of strongly orthogonal positive roots \\
\hline $A_n$&$\beta_i =
\alpha_i + \cdots + \alpha_{n+1-i}; 1\leq i\leq [(n+1)/2]\;{\rm with}\;\beta_{\frac{n+1}{2}}=\alpha_{\frac{n+1}{2}}\;{\rm if }\;n\;{\rm is\;odd}$\\
\hline
\multirow{2}{*}{$B_n$}&$\beta_i = \alpha_{2i-1} + 2\alpha_{2i}+\cdots +2\alpha_n;\,1\leq i\leq [(n+1)/2]\;{\rm with}\;\beta_{\frac{n+1}{2}}=\alpha_n\;{\rm if}\;n\;{\rm is\;odd}$\\
&$\beta_{i'}=\alpha_{2i'-1};\, 1\leq i'\leq [n/2]$\\ \hline
$C_n$&$\beta_i =2\alpha_i+\cdots+2\alpha_{n-1}+\alpha_n;\,1\leq
i\leq n,\;{\rm with}\;\beta_n=\alpha_n$\\ \hline
\multirow{3}{*}{$D_n$}&$\beta_i=\alpha_{2i-1} + 2\alpha_{2i}+\cdots +2\alpha_{n-2}+\alpha_{n-1}+\alpha_n ; 1\leq i\leq [n/2]-1$\\
&{$\beta_{i' }= \alpha_{2i'-1} ; 1\leq i'\leq [n/2]-1$}\\
&{\rm and}\; $\beta_{\frac{n-1}{2}}=\alpha_{n-2}+\alpha_{n-1}+\alpha_n\;{\rm if}\;n\;{\rm is\;odd},\; \beta_{\frac{n}{2}}=\alpha_{n}\,\;{\rm if}\;n\;{\rm is\; even}$\\
&$\beta_{(\frac{n-2}{2})''}=\alpha_{n-1}\;{\rm if}\;n\;{\rm is\;even}$\\
\hline
\multirow{4}{*}{$E_6$}&$\beta_1=\alpha_1+2\alpha_2+2\alpha_3+3\alpha_4+2\alpha_5+\alpha_6$\\
&$\beta_2=\alpha_1+\alpha_3+\alpha_4+\alpha_5+\alpha_6$\\
&$\beta_3=\alpha_3+\alpha_4+\alpha_5$\\
&$\beta_4=\alpha_4$\\ \hline
\multirow{5}{*}{$E_7$}&$\beta_1=2\alpha_1+2\alpha_2+3\alpha_3+4\alpha_4+3\alpha_5+2\alpha_6+\alpha_7$\\
&$\beta_2=\alpha_2+\alpha_3+2\alpha_4+2\alpha_5+2\alpha_6+\alpha_7$\\
&$\beta_3=\alpha_2+\alpha_3+2\alpha_4+\alpha_5$\\
&$\beta_4=\alpha_3$\\
&$\beta_{2'}=\alpha_7,\, \beta_{3'}=\alpha_5,\,\beta_{3''}=\alpha_2$\\ \hline
\multirow{6}{*}{$E_8$}&$\beta_1=2\alpha_1+3\alpha_2+4\alpha_3+6\alpha_4+5\alpha_5+4\alpha_6+3\alpha_7+2\alpha_8$\\
&$\beta_2=2\alpha_1+2\alpha_2+3\alpha_3+4\alpha_4+3\alpha_5+2\alpha_6+\alpha_7$\\
&$\beta_3=\alpha_2+\alpha_3+2\alpha_4+2\alpha_5+2\alpha_6+\alpha_7$\\
&$\beta_4=\alpha_2+\alpha_3+2\alpha_4+\alpha_5$\\
&$\beta_5=\alpha_3$\\
&$\beta_{3'}=\alpha_7,\, \beta_{4'}=\alpha_5,\,
\beta_{4''}=\alpha_2$\\ \hline
\multirow{4}{*}{$F_4$}&$\beta_1 =
2\alpha_1+3\alpha_2+4\alpha_3+2\alpha_4$\\
&$\beta_2=\alpha_2+2\alpha_3+2\alpha_4$\\
&$\beta_3=\alpha_2+2\alpha_3$\\
&$\beta_4=\alpha_2$\\ \hline
\multirow{2}{*}{$G_2$}&$\beta_1 = 3\alpha_1+2\alpha_2$\\
&$\beta_2 = \alpha_1$\\ \hline
\end{tabular}\\
\footnotesize{Table I}
\end{center}

\subsubsection{}\label{heisenbergconditions} We will slightly digress in order to introduce the sets $H_{\beta_K},\, K\in \mathcal K$ and some of their properties that will be useful in the sequel. For $A\subset\Delta$, denote by $-A=\{\alpha\in\Delta\mid-\alpha\in A\}$.

For all $K\in\mathcal K$ set $H_{\beta_K}= \{\gamma\in\Delta_K \mid (\gamma,\,
\beta_K)> 0\}$ and $H_{-\beta_K}=-H_{\beta_K} $. By \cite[Lemma 2.2]{J1}, we have the following :
\begin{enumerate}
\item \label{pos} $H_{\beta_K}=\Delta^+_K\setminus{(\Delta^+_K)}_{\beta_K}$.
\item \label{deco} $\Delta^+ = \bigsqcup\limits_{L\in \mathcal{K}}H_{\beta_L}$ and $\Delta^- = \bigsqcup\limits_{L\in
\mathcal{K}} H_{-\beta_L}.$
\item \label{h} If $\alpha\in H_{\beta_K}\setminus\{\beta_K\}$, then $\beta_K-\alpha\in H_{\beta_K}\setminus\{\beta_K\}.$
\item \label{adclosed} Given $\alpha,\,\alpha'\in H_{\beta_K}$, if $\alpha+\alpha'\in\Delta$, then $\alpha+\alpha'=\beta_K.$
\item \label{prop2} If $\alpha\in H_{\beta_K},\,\alpha'\in H_{\beta_L}$ are such that $\alpha+\alpha'\in \Delta$, then $K\le L$ (resp. $L\le K$) and $\alpha+\alpha'\in H_{\beta_K}$ (resp. $\alpha+\alpha'\in H_{\beta_L}$). In particular, if $\alpha+\alpha'=\beta_K$, then $K=L$ and $\alpha,\, \alpha'\in H_{\beta_K}\setminus \{\beta_K\}$.
\item \label{Hpi} $|H_{\beta_K}\cap \pi|=1$ if $\Delta_K$ is not of type $A$ and $2$ if it is. In fact, \[H_{\beta_i}\cap \pi=\left\{\begin{array}{ll}\{\alpha_i,\, \alpha_{n+1-i}\}, &\textrm{in type A}.\\
\{\alpha_{2i}\}, &\textrm{in type BD}.\\
\{\alpha_i\}, &\textrm{in type C}.\end{array}\right.\]
\end{enumerate}

\subsubsection{}\label{borel} Recall that $\{\varpi_{\alpha}\}_{\alpha\in \pi}$ is
the set of fundamental weights of $\mathfrak{g}$ and $j = -w_0$ is the Dynkin diagram involution.
Let $P^+(\pi):=\bigoplus\limits_{\alpha\in \pi}\mathbb N \varpi_\alpha$ denote the set of dominant weights of $\lieg$ and set $\mathcal B_\pi:=P^+(\pi)\cap \mathbb N\beta_\pi$. Define $P^+(\pi'),\, \mathcal B_{\pi'}$ accordingly. Then $\mathcal B_{\pi}$ is the set of weights of the Poisson semicentre $Sy(\b)$, where $\b=\n\oplus\h$ is the Borel subalgebra.
In this case, $Sy(\b)$ is a polynomial algebra in $\rk\,\g$ generators \cite{J1}. The generators of the free semigroup $\mathcal B_{\pi}$ may be found in \cite[Tables I and II]{J1} and in \cite[Table]{FJ2}.

Set for all $\Gamma\in E$ :
$$\varepsilon_{\Gamma} =
\left\{
\begin{array}{lr} 1/2, \;\;\;\textrm{if }\;\;\Gamma=j\Gamma,\; \sum\limits_{\alpha\in\Gamma}\varpi_{\alpha}\in \mathcal B_\pi\,,\; \sum\limits_{\alpha\in\Gamma\cap \pi'}\varpi_{\alpha}'\in \mathcal B_{\pi'},\\
1,\;\;\; \textrm{otherwise}.
\end{array}
\right.
$$
\\

\begin{theorem} \cite[Thm. 7.3]{FJ2}. If $\varepsilon_\Gamma=1$ for all $\Gamma\in E$, then the algebra $Y(\trp)$ is polynomial in $|E|=\ind\,\trp$ generators. In particular, when $\g$ is simple of type $A$ or $C$, then $\varepsilon_{\Gamma}=1$, for all $\Gamma\in E$.
\end{theorem}\\

\noindent {\bf Remarks.}
\begin{enumerate}
\item The bounds for $Sy(\p_{\pi'})$ that we mentioned in the beginning of \ref{bounds} coincide precisely when $\varepsilon_\Gamma=1$ for all $\Gamma\in E$.
\item Examples where the bounds do not coincide but $Y(\p_{\pi',\,E})$ is still polynomial are known (see for instance \cite{J1} for the Borel case or \cite{J6bis} for the centralizer of the highest root vector, which is a particular truncated parabolic algebra).
\item In \cite{Y} there has been found an example where $Y(\trp)$ is not polynomial. This is the case of the maximal parabolic subalgebra of $\lieg$ of type $E_8$ with $\pi'=\pi\setminus \{\alpha_8\}$ (in the Bourbaki notation \cite[Planche VII]{BOU}). In this case, $\trp=\mathfrak g^{x_{\beta}}$, where $\beta$ is the unique highest root of $\mathfrak g$ and $x_\beta\in \mathfrak g_\beta\setminus \{0\}$.
\end{enumerate}

\section{The construction of an adapted pair}
Recall the definition of an adapted pair of Section \ref{adaptedpair1}. In \cite{J5}, an adapted
pair is constructed for all truncated (bi)parabolic subalgebras of
$\liesl_n$. In the following sections we will construct an adapted
pair - whenever possible - for all maximal truncated parabolic subalgebras in the remaining types when the bounds of \cite{FJ2} coincide.
By a maximal parabolic we mean that $\pi' = \pi\setminus
\{\alpha_s\}$.\\

\noindent {\bf Remark.} The reason that we restrict ourselves to the maximal parabolic case unlike the case of $\liesl_n$ where an adapted pair has been constructed for all parabolic subalgebras is that for $\liesl_n$ the truncated Cartan subalgebra is never too small. By contrast, when $j = \id$ and so $\lieh_E = \lieh'$, the truncated Cartan subalgebra may be very small or even $\{0\}$ in the case of the truncated Borel.
We will see in Section \ref{F}, that even for maximal truncated parabolics with Poisson centre polynomial, adapted pairs may not exist.\\

It is convenient from now on to replace $\trp$ with its
opposed algebra $\trp^-$. Identify $(\trp^-)^*$ with $\trp$ through the
Killing form. Then $\trp^-$ acts on $\trp$ via
the coadjoint action. If $\mathfrak{h}^{\perp}_E$ is the
orthogonal complement of $\mathfrak{h}_E$ in $\mathfrak{h}$ with
respect to the Killing form, then we can identify
$(\mathfrak{p}_{\pi',E}^-)^*$ with
$\mathfrak{g}/(\mathfrak{h}^{\perp}_E\oplus \mathfrak{m}^-)$,
where $\mathfrak{m}^-$ is the nilradical of $\mathfrak{p}^-_{\pi',E}$ and
then the coadjoint action is computed through commutation in
$\mathfrak{g} \mod \mathfrak{h}^{\perp}_E\oplus \mathfrak{m}^-$.

\subsection{Heisenberg sets}\label{heisenberg}
For $A\subset\Delta$, set $\g_A:=\bigoplus\limits_{\alpha\in A}\g_{\alpha}$, which is a subspace of $\g$ and a Lie subalgebra of $\g$ if $A$ is additively closed in $\Delta$.\\

\noindent {\bf Definition.} For any $\gamma\in \Delta$, a subset $\Gamma_\gamma$ of $\Delta$ is called a {\it Heisenberg set of centre $\gamma$} if $\gamma\in\Gamma_{\gamma}$ and for all $\alpha\in \Gamma_\gamma\setminus\{\gamma\}$, there exists a unique $\alpha'\in \Gamma_\gamma$ such that $\alpha+\alpha'=\gamma$.\\

\noindent {\bf Remark.}
For a fixed $\gamma$, there exist several Heisenberg sets $\Gamma_{\gamma}$ of centre $\gamma$. All of them are included in a maximal one, namely the set of all decompositions of $\gamma$ into the sum of two roots in $\Delta$. (Note that an $\langle ij\rangle$-orbit in $\pi$ was also denoted by $\Gamma$; we will be very careful not to have any confusion).\\

\noindent {\bf Example.} Recall the sets $H_{\beta_K},\, K\in \mathcal K$ defined in \ref{heisenbergconditions}. One has that $H_{\beta_K}$ is a Heisenberg set of centre $\beta_K$ by \ref{heisenbergconditions} (3) and actually it is the maximal Heisenberg set of centre $\beta_K$, which is included in $\Delta^+$. Moreover, by \ref{heisenbergconditions} (4), $\g_{H_{\beta_K}}$ is a Heisenberg Lie algebra of centre $\g_{\beta_K}$. For an arbitrary $\Gamma_\gamma$ it is no longer true that $\g_{\Gamma_\gamma}$ is a Heisenberg Lie algebra or even a Lie subalgebra of $\g$; however, inspired by this example, we called the sets $\Gamma_\gamma$ Heisenberg sets.\\

If $\Gamma_\gamma$ is a Heisenberg set of centre $\gamma$, then $-\Gamma_{\gamma}=\Gamma_{-\gamma}$ is a Heisenberg set of centre $-\gamma$. We set $\Gamma^0_{\gamma}:=\Gamma_{\gamma}\setminus\{\gamma\}$. Similarly, we define $H_{\beta_K}^0$.

\subsection{} The following lemma can be found in \cite[Lemma 8.5]{J5}:\\

\begin{lemma}\label{nondegenerate} Let $\Gamma^\pm$ be subsets of $\Delta^\pm$ and
suppose there exist subsets $S^\pm$ of $\Gamma^\pm$ such that $\Gamma^\pm=
\coprod\limits_{\gamma\in S^\pm} \Gamma_{\gamma}$, where $\Gamma_{\gamma}$ is a Heisenberg set of centre $\gamma$, for all $\gamma\in S^\pm$. Set $S=S^+\sqcup S^-$ and $O^\pm=\coprod\limits_{\gamma\in S^\pm} \Gamma_{\gamma}^0$. Set also $\o= \bigoplus\limits_{\gamma\in
S}\lieg_{-\Gamma_{\gamma}^0}=\g_{-O^+}\oplus \g_{-O^-}$.

Assume further that the elements of $S$ are linearly independent and that if
$\alpha,\,\beta\in O^\pm$ are such that
$\alpha+\beta= \gamma$, for some $\gamma\in S^\pm$, then
$\alpha,\,\beta\in\Gamma_{\gamma}^0$. Set \[f = \sum\limits_{\gamma\in S}a_{\gamma}x_{\gamma},\]
where $a_{\gamma}$ are non-zero scalars for all $\gamma\in S$. Then the
bilinear form $\Phi_f : \lieg \times \lieg \rightarrow k$
defined by $\Phi_f(x,\,x') = K(f,\,[x, \,x'])$ for all $x,\, x' \in\lieg$, where $K$ is the Killing form on $\g$, is
non-degenerate on $\o \times
\o$.
\end{lemma}
\smallskip

When $S^+=\beta_{\pi^+}$, $S^-=-\beta_{\pi^-}$ for two subsets $\pi^+,\,\pi^-\subset \pi$ and $\Gamma_\gamma=H_\gamma$ for all $\gamma\in S$, the above result follows from \cite[Remarque 3.9]{TY}.

\subsection{}

\begin{lemma}\label{regulargeneral} Suppose $\pi'\subset \pi$ is a
subset of $\pi$ and let $\liep_{\pi', E}$ be the corresponding
truncated parabolic subalgebra. Let $S^+, \,T^+\subset \Delta^+$ and
$S^-, \,T^-\subset \Delta^-_{\pi'}$ be such that the following conditions hold :
\begin{enumerate}
\item For all $\gamma\in S^\pm$ there exist Heisenberg sets $\Gamma_{\gamma}$ of centre $\gamma$ such that, \[\Gamma^+ : = \Delta^+\setminus T^+=\bigsqcup\limits_{\gamma\in
S^+}\Gamma_{\gamma},\] and \[\Gamma^- : = \Delta^-_{\pi'}\setminus
T^- = \bigsqcup\limits_{\gamma\in S^-}\Gamma_{\gamma}.\]
Set $O^\pm=\bigsqcup\limits_{\gamma\in
S^\pm}\Gamma_{\gamma}^0$.
\item If  $\alpha, \beta\in O^\pm$ are such that
$\alpha + \beta = \gamma$ for some $\gamma\in S^\pm$ then $\alpha,
\beta\in\Gamma_{\gamma}^0$.
\item Set $S=S^+\sqcup S^-$. Then $S|_{\lieh_E}$ is a basis of $\lieh_E^*$.
\item Set $T=T^+\sqcup T^-$. Then $|T|=\ind \trp$.
\end{enumerate}
Then the element $y:=\sum\limits_{\gamma\in S}x_{\gamma}$ is regular and $\mathfrak g_T:=\bigoplus\limits_{\alpha\in T}\g_{\alpha}$ is a complement of the $\ad\, \trp^-$ - orbit of $y$ in $\trp$, that is
$$(\ad \liep^-_{\pi', E})\,y \oplus
\lieg_T = \liep_{\pi', E}.$$
\end{lemma}

\begin{proof}
The proof is similar to that of \cite[Theorem 8.6]{J5}. We will give it for completeness.
Retain the notations of Lemma \ref{nondegenerate}. Note that $\g_{O^+}\oplus \g_{O^-}$ identifies with $\o^*$ via the Killing isomorphism.
Condition (1) implies the equalities $\p_{\pi',\,E}=\h_E\oplus\o^*\oplus\g_S\oplus\g_T$ and $\p_{\pi',\,E}^-=\h_E\oplus\o\oplus\g_{-S}\oplus\g_{-T}$.
Conditions (2) and (3) and the expansion of $y$ imply by Lemma~\ref{nondegenerate} that $\Phi_y$ is non-degenerate on
$\o\times\o$. Then $\o^*\subset(\ad\,\o)\,y\subset(\ad\,\p_{\pi',\,E}^-)\,y$. 
Condition (3) implies that $\g_S=(\ad\,\h_E)\,y$ and that $\h_E=(\ad\,\g_{-S})\,y +o^*\oplus\g_S\oplus\g_T$.
Hence $\p_{\pi',\,E}=\h_E\oplus\o^*\oplus\g_S\oplus\g_T\subset(\ad\,\p_{\pi',\,E}^-)\,y+\g_T$. Finally, condition (4) forces the latter sum to be direct, since $\dim\g_T=\ind \trp\le \operatorname{codim} (\ad \liep^-_{\pi', E})\,y$.
\end{proof}

\subsection{}

The corollary below follows from Theorem~\ref{adaptedpair1} and Lemma~\ref{regulargeneral} :\\

\begin{corollary}\label{sliceheisenberg} Suppose that the hypotheses of the previous lemma hold for $\trp$ and define $h\in \lieh_E$ by setting $\alpha(h)=-1$ for all $\alpha\in S$. Then $(h,\,y)$ is an adapted pair for $\trp$. If, in addition, $Y(\trp)$ is polynomial in $\ind \trp$ generators, then $y+\g_T$ is a slice for the coadjoint action of $\trp^-$.
\end{corollary}\\


\subsubsection{}\label{STcascade} In the following sections we construct an adapted pair for maximal parabolic subalgebras for which $Y(\trp)$ is known to be polynomial, that is when Theorem ~\ref{borel} applies. By the above, it is enough to choose, when possible, two subsets $S, \,T$ of $\Delta^+\sqcup \Delta^-_{\pi'}$ which satisfy the conditions of Lemma \ref{regulargeneral}.

As we will see, in most cases, we choose sets $S^+,\, T^+$, with $T^+\subset \pi$ such that $S^+\sqcup T^+=\beta_\pi$ and $S^-,\, T^-$, with $T^-\subset -\pi'$ such that $S^-\sqcup T^-=-\beta_{\pi'}$. Then for all $\beta\in S$, we take $\Gamma_\beta=H_\beta$, as defined in \ref{heisenbergconditions}. By the easy observation that if $\beta\in \pi$, then $H_\beta=\{\beta\}$, we conclude that with these choices condition (1) of  Lemma \ref{regulargeneral} is satisfied by \ref{heisenbergconditions} (2). Moreover, condition (2) follows by \ref{heisenbergconditions} (5). Hence in these cases, it only remains to verify conditions (3) and (4).

As adapted pairs were already constructed for all truncated biparabolic subalgebras of $\mathfrak{sl}_n$ \cite{J5}, we will not consider this case.

\section{Type $B$ }\label{casB}

Let $\lieg$ be simple of type $B_n$ ($n\ge 2$) and $\pi^{\prime} = \pi\setminus
\{\alpha_s\}$, for some $1\leq s\leq n$. The truncated parabolic
$\liep_{\pi', E}$ is then equal to $\liep_{\pi', E} = \lien\oplus
\lieh'\oplus \lien_{\pi'}^-$, where recall that $\lieh'$ is the Cartan
subalgebra of the derived algebra of $\liep_{\pi'}$. Its Levi
factor $\liel_{\pi'}$ is the product of two simple Lie algebras,
one of type $A_{s-1}$ and the other of type $B_{n-s}$. In
particular, if $s = 1$ one has that $\liel_{\pi'} =
\mathfrak{so}_{2n-1}$ and if $s= n,\,\liel_{\pi'} = \liesl_{n}$.

\subsection{Orbits}
Since $w_0 = -\id$ one has that $j = \id$ and $i$ is as follows :
$$i(\alpha_t) = \left\{
\begin{array}{lr}\alpha_{s-t}, &1\leq t\leq s-1\\
\alpha_t, &s\leq t\leq n
\end{array}\right.
$$
Hence the $\langle ij\rangle$ - orbits of $\pi$ are $\Gamma_t = \{\alpha_t,\,
\alpha_{s-t}\}, 1\leq t\leq [\frac{s}{2}]$ and $\Gamma_t =
\{\alpha_t\},\,s\leq t\leq n$.

Then $\mid\! E\!\mid=[\frac{s}{2}]+n-s+1$.

By Theorem~\ref{borel}, if $s$ is odd, $Y(\p_{\pi'\,E})$ is polynomial, since $\varepsilon_{\Gamma}=1$ for all $\Gamma\in E$. If $s$ is even the theory of \cite{FJ2} does not guarantee us that $Y(\trp)$ is polynomial. However, when $s=2$, then $\trp$ equals the centralizer of the nilpotent element corresponding to a non-zero highest root vector $x_{\beta_1}$, that is $\trp=\lieg^{x_{\beta_1}}$, and by \cite{PPY} $Y(\trp)$ is polynomial.
Moreover, in this case an adapted pair was constructed in \cite{J6bis}.

Below, we will construct an adapted pair when $s$ is odd and will give the adapted pair of \cite{J6bis} for $s=2$. 

\subsection{Construction of an adapted pair.}

Here we assume that $\pi'=\pi\setminus\{\alpha_s\}$ with $s$ odd. Note that in this case $\ind\, \trp=\mid\! E\!\mid=n-\frac{s-1}{2}$.

\subsubsection{The sets $S$ and $T$}\label{casBnotations}
Recall the maximal set $\beta_{\pi}$ of strongly orthogonal roots we introduced in section \ref{cascade} and in particular Table I. Denote by
$\beta_{\pi}^0$ the set $\beta_{\pi}^0 = \beta_{\pi}\setminus (\beta_{\pi}\cap \pi)$.

The subset $\pi'$ is a union of two connected components, $\pi_1'$ of type $A_{s-1}$ and $\pi_2'$ of type $B_{n-s}$. One has that $\beta_{\pi'_1}=\{\beta_i':=\alpha_i+\alpha_{i+1}+\cdots  +\alpha_{s-i}\,|\, 1\le i\le \frac{s-1}{2}\}$. In particular, since $s$ is odd $\beta_{\pi_1'}^0 = \beta_{\pi_1'}$.

Also, $\beta_{\pi_2'}^0=\{\beta_i'':=\alpha_{s+2i-1}+2\alpha_{s+2i}+\ldots+ 2\alpha_n\,|\,1\le i\le[\frac{n-s}{2}]\}$ and $\beta_{\pi_2'}\cap {\pi'_2}=\{\alpha_{s-1+2i}\,|\,1\le i\le[\frac{n+1-s}{2}]\}$.

For $S$ we choose $S =S^+\sqcup S^-$, where $S^+=\beta_{\pi}^0$ and $S^-=(-\beta_{\pi_1'})\cup(-\beta_{\pi_2'}^0)$.
%
For $T$ we choose $T = T^+\sqcup T^-$ where $T^+=\beta_{\pi}\cap \pi$ and $T^-= -(\beta_{\pi_2'}\cap \pi_2')$.

\subsubsection{\it Choice of Heisenberg sets.}
For all $\beta\in S$ we take $\Gamma_\beta=H_\beta$, as defined in \ref{heisenbergconditions}. Notice that $S^+\sqcup T^+=\beta_\pi$ and $S^-\sqcup T^-=-\beta_{\pi'}$. By \ref{STcascade}, conditions (1) and (2) of  Lemma~\ref{regulargeneral} are satisfied.
%
%
%
%
%
Condition (4) is also satisfied, since $\mid\!T\!\mid=n-\frac{s-1}{2}=\ind\p_{\pi',\,E}$. It remains to verify condition (3).

\subsubsection{}\label{lemmebaseB}
\begin{lemma} \label{typeBsupport} For the above choice of $S$, one has that $S|_{\lieh _E}$ is a basis for
$h^*_E$.
\end{lemma}
\begin{proof}
Let $S=\{s_j\,|\,1\le j\le n-1\}$ and take for a basis of $\mathfrak h_E$ the set $\{\alpha_i^\vee\,|\,1\le i\le n,\, i\ne s\}$. It is enough to show that the determinant of the matrix $(\alpha_i^\vee(s_j))_{i, \,j}$ is non-zero. 


First of all notice that for all $i$, with $1\le i\le [n/2]$, $\beta_i=c_i\varpi_{2i}-\varpi_{2i-2}$, where $c_{n/2}=2$ when $n$ is even and $c_i=1$ otherwise. (Here we have set $\varpi_0=0$.) Similarly, $\beta''_i=d_i\varpi_{s+2i}-\varpi_{s+2i-2}$, with $d_{\frac{n-s}{2}}=2$ if $n$ is odd and $d_i=1$ otherwise. Then up to a sign we have:
$\det\,(\alpha_i^\vee(s_j))=\det\,\left(\begin{array}{cc}A&B\\0&C\end{array}\right)$, where $A=(\alpha_{2i}^\vee(\beta_i))$ is an upper triangular matrix with $c_i,\, 1\le i\le [n/2]$ on the diagonal. Then up to a sign $\det\, (\alpha_i^\vee(s_j))=c_{[n/2]}\det\, C$.

On the other hand $C=\left(\begin{array}{cc}D&0\\0&F\end{array}\right)$, where $D=(\alpha_{s+2i}^\vee(\beta_j''))$ is an upper triangular matrix with $d_i$ on the diagonal, hence $\det\, D=d_{[\frac{n-s}{2}]}\det\, F$.

Finally, $F=-(\alpha_{2i-1}^\vee(\beta_j'))$ and each of the $\alpha_{2i-1}$,\, $1\le i\le \frac{s-1}{2}$, lies in exactly one $H_{\beta_j'}$ by~\ref{heisenbergconditions}~(6). Recall that $\{\beta'_i\,|\,1\le i\le \frac{s-1}{2}\}$ is the set $\beta_{\pi_1'}$ of strongly orthogonal roots of $\pi_1'$, which is of type $A_{s-1}$. In particular, they are totally ordered by the order defined in \ref{cascade}. With the notations of \ref{cascade}, $\Delta^+_{i+1}=\Delta^+_{\beta'_{i}}$ and $\Delta^+_{i+1}\subset \Delta^+_{i}$ for all $i$, with $1\le i\le \frac{s-1}{2}$ (where we set $\Delta^+_1=\Delta^+_{\pi'_1}$). Moreover, by \ref{heisenbergconditions} (1) and the above, $H_{\beta'_i}\subset \Delta^+_{\beta'_{i-1}}=\Delta^+_{i}$. Let $\alpha_{t_j}\in H_{\beta_j'}$, with $\{\alpha_{t_j}\,|\,1\le j\le \frac{s-1}{2}\}=\{\alpha_{2i-1}\,|\,1\le i\le \frac{s-1}{2}\}$. Then up to a sign, $\det\, F=\det\, (\alpha_{t_i}^\vee(\beta_j'))$ and the latter is an upper triangular matrix with $1$ on the diagonal. In particular, $\det\, F\ne 0$.
\end{proof}


\subsubsection{}
Now take $\pi'=\pi\setminus\{\alpha_2\}$; as we already said, an adapted pair has been constructed in \cite{J6bis}.

The sets $S$ and $T$ are given in~\cite[Tables]{J6bis} for $B_n$, $n\ge 3$ as follows: \[S=\{\alpha_2,\,\alpha_3,\,\ldots,\,\alpha_{n-1},\,\alpha_1+\alpha_2+\ldots+\alpha_n\}\] and \[T=\{\alpha_i+2(\alpha_{i+1}+\ldots+\alpha_n),\, 1\le i\le n-1,\,-\alpha_1\}.\]


\section{Type $D$}

As expected, type $D$ is almost identical with $B$. We will give the details for completeness.

Let $\g$ be simple of type $D_n$ ($n\ge 4$) and $\pi'=\pi\setminus\{\alpha_s\}$ for some $1\le s\le n$.
If $1\le s\le n-3$, then the Levi factor $\l_{\pi'}$ of the truncated parabolic subalgebra $\p_{\pi',\,E}$ is the product
of two simple Lie algebras, one of type $A_{s-1}$ and another of type $D_{n-s}$ (in particular, if $s=1$, $\l_{\pi'}$
is a simple Lie algebra of type $D_{n-1}$).

If $s=n-2$, then the Levi factor $\l_{\pi'}$ is the product of three simple Lie algebras, of types $A_{n-3}$, $A_1$ and $A_1$.

Finally, if $s=n-1$ or $s=n$, then $\l_{\pi'}$ is a simple Lie algebra of type $A_{n-1}$.

\subsection{Orbits.}
Recall that $j=\id$ if $n$ is even and, if $n$ is odd, $j$ is the involution of $\pi$ that interchanges $\alpha_{n-1}$ and $\alpha_n$ and fixes the rest of the simple roots.
By Theorem~\ref{borel} one can show that $Y(\p_{\pi',\,E})$ is polynomial (that is $\varepsilon_{\Gamma}=1$ for all $\Gamma\in E$) if $s$ is odd or $s=n-1$ and $s$ is even. Since the cases $s=n,\, n-1$ are symmetric, we may consider that $s$ is odd. Note that then $n, \,n-s$ have different parity.

We will compute the $\langle ij\rangle$ - orbits in $\pi$ when $s$ is odd.

Assume first that $1\le s\le n-2$. The $\langle ij\rangle$ - orbits of $\pi$ are $\Gamma_t=\{\alpha_t,\,\alpha_{s-t}\}$, $1\le t\le\frac{s-1}{2}$, $\Gamma_t=\{\alpha_t\}$, $s\le t\le n-2$ and $\Gamma_{n-1}=\{\alpha_{n-1}, \alpha_n\}$. Then $|E|=n-\frac{s+1}{2}$.

Assume that $s=n$ and is odd. The $\langle ij\rangle$ - orbits in $\pi$ are $\Gamma_t=\{\alpha_t,\,\alpha_{n-t}\}$, $2\le t\le \frac{n-1}{2}$ and $\Gamma_1=\{\alpha_1,\,\alpha_{n-1},\,\alpha_n\}$. Then $|E|=\frac{n-1}{2}$.

In \cite{PPY}, the authors show that for $s=2$, $Y(\trp)$ is also polynomial; an adapted pair is constructed in \cite{J6bis} in this case - we give it at the end of this section.


\subsection{Construction of an adapted pair.}

\subsubsection{The sets $S$ and $T$.}
Assume first that $1\le s\le n-2$ and $s$ is odd. Denote by $\pi'_1$ the connected component of $\pi'$ of type $A_{s-1}$ and by $\pi'_2$ the connected component of $\pi'$ of type $D_{n-s}$.
Recall the sets $\beta_{\pi}$ of strongly orthogonal positive roots (see Table I) and $\beta^0_{\pi}$ as defined in~\ref{casBnotations}. Note in particular that $|\beta_\pi^0|=\frac{n-1}{2}$ if $n$ odd and $|\beta_\pi^0|=\frac{n}{2}-1$ if $n$ even and that $n, n-s$ have different parity.

Set $\beta_\pi^0=\{\beta_i\,|\, 1\le i\le \frac{n-1}{2}\}$ if $n$ is odd and $\beta_\pi^0=\{\beta_i\,|\, 1\le i\le \frac{n}{2}-1\}$ if $n$ is even. Then, in the former case, we have $\beta^0_{\pi_2'}=\{\beta_i''\,|\,1\le i\le \frac{n-s}{2}-1\}$ and in the latter $\beta^0_{\pi_2'}=\{\beta_i''\,|\,1\le i\le \frac{n-s-1}{2}\}$. The roots $\beta''_i$ are given by the formulae of Table I for type $D$ by shifting $i\rightarrow s+i$. Also $\beta_{\pi_1'}=\beta^0_{\pi'_1}=\{\beta'_i\,\mid 1\le i\le \frac{s-1}{2}\}$, where the $\beta'_i$ are given by the formulae of Table I for Type $A$ by setting $n=s-1$. 

Set now $S^+=\beta^0_{\pi}$ if $n$ is odd and $S^+=\beta^0_{\pi}\sqcup\{\alpha_n\}$ if $n$ is even. Set also $S^-=(-\beta^0_{\pi'_1})\sqcup(-\beta^0_{\pi'_2})$ if $n$ is even and $S^-=(-\beta^0_{\pi'_1})\sqcup(-\beta^0_{\pi'_2})\sqcup\{-\alpha_n\}$ if $n$ is odd.

Set $T^+=\beta_\pi\cap \pi$ if $n$ is odd and $T^+=(\beta_\pi\cap \pi)\setminus \{\alpha_n\}$ if $n$ is even. Set also $T^-=-(\beta_{\pi_2'}\cap \pi'_2)\setminus \{-\alpha_n\}$ if $n$ is odd and $T^-=-(\beta_{\pi_2'}\cap \pi'_2)$ if $n$ is even. Take $T=T^+\sqcup T^-$. In both cases, $|T|=n-\frac{s+1}{2}=\ind\,\trp$. \\

Assume now that $s=n$ and is odd. Then $\pi'$ is of type $A_{n-1}$. Let $\{\beta_i\,|\,1\le i\le\frac{n-1}{2}\}$ be the elements in $\beta_\pi$ as noted in Table I and $\beta_{\pi'}=\{\beta'_i:=\alpha_i+\ldots+\alpha_{n-i}\,|\,1\le i\le\frac{n-1}{2}\}$.

We choose $S^+=\{\beta_i\,|\,1\le i\le\frac{n-1}{2}\}=\beta^0_{\pi}$, $S^-=-\beta_{\pi'}$ and $S=S^+\sqcup S^-$.
%
%
For $T^\pm$ we choose $T^+=\beta_{\pi}\cap \pi=\{\alpha_{2i-1},\,1\le i\le\frac{n-1}{2}\}$, $T^-=\emptyset$ and $T=T^+$.

Then $\mid T\mid=\frac{n-1}{2}=\mid E\mid=\ind \p_{\pi',\,E}$.

\subsubsection{\it Choice of Heisenberg sets.}
For each $\beta\in S$, set $\Gamma_{\beta}=H_{\beta}$ with the notation of \ref{heisenbergconditions}. As in type $B$, we have that $S^+\sqcup T^+=\beta_\pi$ and $S^-\sqcup T^-=-\beta_{\pi'}$. Then conditions (1) and (2) of Lemma
\ref{regulargeneral} are immediate, condition (4) is verified above. Finally, condition (3) also follows by arguments similar to the proof of Lemma~\ref{lemmebaseB}.

\subsubsection{} Finally take $\pi'=\pi\setminus\{\alpha_2\}$, and $n\ge 4$, which is the case when the truncated parabolic subalgebra coincides with the centralizer of the highest root vector. An adapted pair has been constructed in \cite{J6bis}. The sets $S$ and $T$ are given in~\cite[Tables]{J6bis} for $D_n$, $n\ge 4$ as follows:

\[S=\{\alpha_2,\,\alpha_3,\ldots,\,\alpha_{n-2},\,\alpha_1+\alpha_2+\ldots+\alpha_{n-1},\,\alpha_1+\alpha_2+\ldots+\alpha_{n-2}+\alpha_n\}\]
and \[T=\{\alpha_i+2\alpha_{i+1}+\ldots+2\alpha_{n-2}+\alpha_{n-1}+\alpha_n,\,1\le i\le n-2,\,\alpha_n,\,-\alpha_1\}.\]

\section{Type $C$}\label{C}
Let $\lieg$ be a simple Lie algebra of type $C_n$ ($n\ge 3$) and $\pi' =
\pi\setminus \{\alpha_s\}$, with $1\le s\le n$. Then $\pi'=\pi'_1\cup \pi'_2$, where $\pi'_1$ is of type $A_{s-1}$ and $\pi'_2$ is of type
$C_{n-s}$. The Levi factor $\liel_{\pi'}$ of the truncated
parabolic $\liep_{\pi',\, E}$ is the product $\mathfrak{l}_{\pi'} = \liesl_s\times
\mathfrak{sp}_{2(n-s)}$. By Theorem \ref{borel}, $Y(\trp)$ is polynomial in $\ind\, \trp$ generators for all $s$, with $1\le s\le n$.

\subsection{Orbits}
The involutions of the Dynkin diagram $j$ and $i$ are given by $j
= \id$ and :
$$i(\alpha_t) = \left\{
\begin{array}{lr}\alpha_{s-t}, &1\leq t\leq s-1.\\
\alpha_t, &s\leq t\leq n.
\end{array}\right.
$$
Hence the $\langle ij\rangle=\langle i\rangle$ - orbits of the Dynkin diagram are
$$\Gamma_t = \left\{
\begin{array}{lr}\{\alpha_t,\, \alpha_{s-t}\},& 1\leq t\leq [s/2].\\
\{\alpha_t\},& s\leq t\leq n.
\end{array}\right.$$
They are $|E|= n-s+1 + [s/2]$ in number, hence $\ind\,\trp = n-s+1 + [s/2]$.

\subsection{Construction of an adapted pair}

\subsubsection{The sets $S$ and $T$}\label{sectionsupporty}
In types $B_n$ and $D_n$ we had that
\[ S\cup T = \beta_{\pi}\cup (-\beta_{\pi'}).
\] In type $C_n$ such a choice would not work; notice in particular that $\beta_{\pi_2'}\subset \beta_{\pi}$.

Recall that $\beta_{\pi_1'}^0$
denotes the set $\beta_{\pi_1'}\setminus (\beta_{\pi'_1}\cap{ \pi'_1})$.

Denote by $\beta_i$ the elements in $\beta_{\pi}$, and $\beta'_i$ the elements in $\beta_{\pi'_1}$.
For all $i$ with $1\leq i\leq n-1$ set $\gamma_i = \beta_i - \alpha_i$.

We will distinguish the following two cases :

(i) If $s$ is odd, we choose

$$\begin{array}{l}\label{supporty1} S^+ = \{\gamma_{2i-1} ; 1\leq i\leq [n/2]\},\\
S^-=\{-\gamma_{2j}; (s+1)/2\leq j\leq [(n-1)/2]\}\cup
(-\beta_{\pi'_1})
\end{array}$$
and
$$\begin{array}{l}\label{supportV1} T^+ =
\{\beta_{2i-1}; 1\leq i\leq [(n+1)/2]\},\\
T^-=\{-\beta_{2j}; (s+1)/2\leq
j\leq [n/2]\}.
\end{array}
$$

(ii) If $s$ is even set $t:=[s/4]$ and choose
\begin{eqnarray*}\label{supporty2}S^+ &=& \{\gamma_{2i-1} ; \, 1\leq i\leq t,\, \beta_{2t+1}, \gamma_{2j};\,t+1\leq j\leq [(n-1)/2]\}, \\
S^- &=& \{-\gamma_{2j-1}; s/2+1\leq j\leq [n/2]\}\cup
(-\beta_{\pi'_1}^0)
\end{eqnarray*}
and
\begin{eqnarray*}\label{supportV2} T^+ &=& \{\beta_{2i-1} ;\, 1\leq i\leq t,\, \beta_{2j} ;\, t+1\leq j\leq [n/2]\},\\
T^- &=& \{-\beta_{2j-1};\, s/2+1\leq j\leq [(n+1)/2],
-\beta_{s/2}'\}.
\end{eqnarray*}

\subsubsection{Choice of Heisenberg sets} Set $S=S^+\sqcup S^-$. For every $\gamma\in S$, we will define a Heisenberg set of centre $\gamma$, $\Gamma_\gamma$.

For all $\gamma\in S\cap (\beta_{\pi}\cup (-\beta_{\pi'}))$, we set $\Gamma_\gamma=H_\gamma$. For the rest of the roots in $S$, namely the $\gamma_i$, (resp. $-\gamma_i$) we define $\Gamma_{\gamma_i}:=H_{\beta_i}^0\sqcup H_{\beta_{i+1}}$ (resp. $\Gamma_{-\gamma_i}=H_{-\beta_i}^0\sqcup H_{-\beta_{i+1}}=-\Gamma_{\gamma_i}$). It is better to view these sets schematically, using certain (shifted) Young tableaux that we define right below. These diagrams were used in \cite{Shi} for a different purpose.

We display the positive roots in a shifted diagram $T(C_n)$ of
shape $(2n-1,\, 2n-3,\, \dots,\, 1)$, that is we assign to each box $(i,\,
j)$ of $T(C_n)$ the positive root $t_{i,\, j}$, where :
$$t_{i, \,j} = \left\{
\begin{array}{lr}\alpha_i + \cdots + 2(\alpha_j + \cdots +\alpha_{n-1}) + \alpha_n, &i\leq j\leq n-1.\\
\alpha_i + \cdots + \alpha_{2n-j}, & n\leq j\leq 2n-i.
\end{array}\right.$$
For example, the diagram for $C_3$ is :
\begin{center}
\begin{tabular}{lllllllll}
$2\alpha_1 + 2\alpha_2 + \alpha_3$&&$\alpha_1 + 2\alpha_2 +
\alpha_3$&&$\alpha_1 + \alpha_2 + \alpha_3$&&$\alpha_1
+\alpha_2$&&$\alpha_1$\\
&&$2\alpha_2 + \alpha_3$&&$\alpha_2 + \alpha_3$&&$\alpha_2$&\\
&&&&$\alpha_3$&&&&
\end{tabular}
\end{center}
Notice that for all $i$, with $1\leq i\leq n$, the $i$ - th line of $T(C_n)$
is the Heisenberg set $H_{\beta_i}$, with the centre $\beta_i$
lying on the $(i,\, i)$ box, i.e. $H_{\beta_i} = \{t_{i,\, j}\mid i\leq j\leq
2n-i\}$ and $t_{i,\, i} = \beta_i$. The right corners of the diagram
correspond to the simple roots.

One has that $\gamma_i=\beta_i - \alpha_i=t_{i,\, i+1}$ for all $i$, with
$1\leq i\leq n-1$. Then $\Gamma_{\gamma_i}$ is the set
$\Gamma_{\gamma_i} = \{t_{\ell,\, m} \mid\ell \in\{ i,\, i+1\},\, m\geq i+1\}$. We will show that it is a Heisenberg set of centre $t_{i,\,
i+1}$. Then of course, $\Gamma_{-\gamma_i}=-\Gamma_{\gamma_i}$ will be a Heisenberg set of centre $-\gamma_i$. \\

\begin{lemma}\label{T(C)} Let $\Gamma_{\gamma_i}$ be the set
$\Gamma_{\gamma_i} = \{t_{\ell,\, m} \mid \ell \in\{i,\, i+1\},\,m\geq i+1\}$.
For all $j$, with $i+1\leq j\leq 2n-i-1$, one has :
\[t_{i+1,\, j}+t_{i,\, 2n+1-j} = t_{i,\, i+1}.\]
In particular, $\Gamma_{\gamma_i}$ is a Heisenberg set of centre $t_{i,\, i+1} (=\gamma_i)$.
\end{lemma}
\begin{proof}
For $j\leq n-1$ one has :
\begin{eqnarray*}t_{i+1,\, j}+t_{i, \,2n+1-j} &=& \alpha_{i+1} + \cdots
+2(\alpha_j + \cdots +\alpha_{n-1})+\alpha_n + \alpha_i + \cdots
+\alpha_{j-1} \\
&=& \alpha_i +2(\alpha_{i+1}+\cdots +\alpha_{n-1}) +\alpha_n \\
&=& t_{i,\, i+1}.
\end{eqnarray*}
For $j = n$ we have :
\[t_{i+1,\, n} + t_{i, \,n+1} = (\alpha_{i+1}+\cdots +\alpha_n) +(\alpha_i+\cdots+\alpha_{n-1}) = t_{i,\, i+1}.\]
The cases $j=n+1$ and $j\ge n+2$ follow by similar calculations. We conclude that for every $t\in \Gamma_{\gamma_i}\setminus \{\gamma_i\}$, there exists a (unique) $t'\in \Gamma_{\gamma_i}\setminus \{\gamma_i\}$, such that $t+t'=\gamma_i$ and $\Gamma_{\gamma_i}$ is a Heisenberg set.
\end{proof}

\subsubsection{Decomposition into Heisenberg sets} We will show that the sets
$\Gamma^\pm:=\bigsqcup\limits_{\gamma\in S^\pm}\Gamma_{\gamma}$ defined in the previous section complement
$T^\pm$ in $\Delta^+$ (resp. $\Delta_{\pi'}^-$), i.e. \[\Delta^+ = \Gamma^+\sqcup T^+,\quad
\Delta_{\pi'}^-=\Gamma^-\sqcup T^-.\] This easily follows by the decomposition in~ \ref{heisenbergconditions}~(\ref{deco}). 
Indeed, take for example the case $s$ odd. Then $\Gamma^+=\bigsqcup\limits_{i=1}^{[n/2]}(H_{\beta_{2i-1}}^0\sqcup H_{\beta_{2i}})$. Clearly, a complement of this in $\Delta^+$ is $T^+$.

Similarly, $$\Gamma^-=\left(\bigsqcup\limits_{i=1}^{(s-1)/2}H_{-\beta'_i}\right)\sqcup \left(\bigsqcup\limits_{i=(s+1)/2}^{[(n-1)/2]}(H_{-\beta_{2i}}^0\sqcup H_{-\beta_{2i+1}})\right).$$ Again, a complement of this set in $\Delta_{\pi'}^-$ is $T^-$.

The case $s$ even follows similarly.

\subsubsection{}\label{C2} We will show that condition (2) of Lemma \ref{regulargeneral} holds.

Set $O^\pm=\coprod\limits_{\gamma\in S^\pm} \Gamma_{\gamma}^0$.

\begin{lemma} If any two roots $\alpha,\, \beta\in O^\pm$ satisfy $\alpha+\beta =
\gamma$, for some $\gamma\in S^\pm$ then $\alpha,\,
\beta\in \Gamma_{\gamma}^0$.
\end{lemma}

\begin{proof} We will prove the statement for $O^+ $. Suppose that $\alpha + \beta = \gamma$ and
$\gamma\in S^+$. Since $\Gamma_\gamma$ is a Heisenberg set, it is enough to show that $\alpha\in \Gamma_{\gamma}^0$ (or $\beta\in \Gamma_\gamma^0$). Notice that the elements of $S^+$ are all of the form $\gamma_i$ or $\beta_i\in \beta_{\pi}$. In the latter case, lemma follows by \ref{heisenbergconditions} (\ref{prop2}).

Let us suppose that $\gamma=\gamma_i=\beta_i-\alpha_i$. Then $\gamma\in H^0_{\beta_i}$ and again by \ref{heisenbergconditions}~(\ref{prop2}), $\alpha\in H_{\beta_i}$ (or $\beta\in H_{\beta_i}$). Let us assume the former holds.

Since all roots $\gamma,\, \alpha,\, \beta$ are positive, we have $\gamma >\alpha$ (here $>$ denotes the usual partial order of the root
lattice). Since both $\gamma,\, \alpha$ are in $H_{\beta_i}$ and $\gamma\in H^0_{\beta_i}$, one has that $\alpha\in
H_{\beta_i}^0$. But $H_{\beta_i}^0\subset \Gamma_{\gamma}$, hence $\alpha\in
\Gamma_{\gamma}$ and then $\alpha\in\Gamma_{\gamma}^0$.
\end{proof}

\subsubsection{}
\begin{lemma} \label{typeCsupport}One has that $S|_{\lieh _E}$ is a basis for
$h^*_E$.
\end{lemma}
\begin{proof} We take for a basis of $\lieh_E$ the set $\{\alpha_i^\vee\,|\, 1\le i\le n,\, i\ne s\}$. We need to show that if $S=\{s_j\,|\,1\le j\le n-1\}$ one has that $\det\, \left(\alpha_i^\vee(s_j)\right)_{i, j}\ne 0$. We may order the $s_j\in S$ and the $\alpha_i^\vee\in h^*_E$ in a way such that the matrix $\left(\alpha_i^\vee(s_j)\right)_{i, j}$ is upper triangular with diagonal elements non-zero. First notice that $\gamma_i=\varpi_{i+1}-\varpi_{i-1}$, for all $i$, with $1\le i\le n-1$, where we have set $\varpi_0=0$.

Assume that $s$ is odd.

By rearranging the rows, we may write the matrix $\left(\alpha_i^\vee(s_j)\right)_{i, j}$ as $\left(\begin{array}{cc}A&B\\0&C\end{array}\right)$ by taking $A$ to be the $[n/2]\times [n/2]$ matrix $A=(\alpha_{2i}^\vee(\gamma_{2j-1}))_{i,\, j=1}^{[n/2]}$. Then $A$ is upper triangular and has $1$ on the diagonal, hence $\det A=1$ and $\det (\alpha_i^\vee(s_j))_{i,\, j}=\det C$. Similarly, we may write $C$ as $C=\left(\begin{array}{cc}D&0\\0&F\end{array}\right)$, with $D$ the $[(n-s)/2]\times [(n-s)/2]$ matrix $D=-(\alpha_{s+2i}^\vee(\gamma_{s+1+2(j-1)}))_{i, j=1}^{[(n-s)/2]}$. Then $\det D=(-1)^{[(n-s)/2]}$, hence $\det C=\det F$ up to a sign. Finally, $F=-\left(\alpha^\vee_{2i-1}(\beta'_j)\right)_{i,\, j=1}^{(s-1)/2}$. 
By \ref{heisenbergconditions} (6), one observes that each of the simple roots $\alpha_{2i-1},\, 1\le i\le (s-1)/2$ lies in exactly one $H_{\beta'_j}$. A similar argument as in the end of the proof of Lemma~\ref{lemmebaseB} allows us to conclude that $\det F\ne 0$.

Assume that $s$ is even and note that $\beta_{2t+1}=2\varpi_{2t+1}-2\varpi_{2t}$. We may again write the matrix as $\left(\begin{array}{cc}A'&B'\\0&C'\end{array}\right)$ by taking $A'$ to be the $t\times t$ matrix $A'=(\alpha_{2i}^\vee(\gamma_{2j-1}))_{i,\, j=1}^{t}$. Then $A'$ is upper triangular and has $1$ on the diagonal, hence $\det A'=1$ and $\det (\alpha_i^\vee(s_j))_{i,\, j}=\det C'$. Similarly, we may write $C'$ as $C'=\left(\begin{array}{cc}D'&E'\\0&F'\end{array}\right)$, with $D'$ the matrix $D'=(\alpha_{2i-1}^\vee(\gamma_{2j-2}))_{i,\, j=t+1}^{[(n+1)/2]}$, where here we have set $\gamma_{2t}:=\beta_{2t+1}$. Then $\det D'=2$ hence $\det C'=2\det F'$. Finally, by reasoning as before, we obtain $\det F'\ne 0$.

\end{proof}

\subsubsection{} Finally, one may observe that in both cases, the set $T$ has cardinality equal to $|T|=\ind\,\trp=n-s+1+[\frac{s}{2}]$. We conclude that all conditions of Lemma \ref{regulargeneral} are satisfied.

\section{Type $E_6$}\label{E6}

Let $\g$ be of type $E_6$ and let $\pi'=\pi\setminus \{\alpha_s\}$, for $1\le s\le 6$, $\trp$ the corresponding truncated parabolic subalgebra. The involution $j$ on $\pi$ is given as follows:
$$j(\alpha_1)=\alpha_6,\, j(\alpha_2)=\alpha_2,\, j(\alpha_3)=\alpha_5, \, j(\alpha_4)=\alpha_4.$$
By Theorem~\ref{borel} $Y(\trp)$ is a polynomial algebra for $s=3,\, 4, \,5$. By \cite{PPY}, $Y(\trp)$ is polynomial for $s=2$ - an adapted pair has been constructed in this case in \cite{J6bis}. In the sections below, we construct an adapted pair for all $\trp$, for $s=3,\,4,\,5$ and we give the construction of \cite{J6bis} for $s=2$, for the sake of completeness. Note that cases $s=3$ and $s=5$ are symmetric, thus it is enough to consider one of them - we treat case $s=3$ in Section \ref{E6-35}.

Recall that we are using the notations of \cite[Planches I-IX]{BOU} and in particular we denote by $\begin{array}{c}acdef\\b\end{array}$ the root $a\alpha_1+b\alpha_2+c\alpha_3+d\alpha_4+e\alpha_5+f\alpha_6$. If $b=0$, we simply write $acdef$ instead of $\begin{array}{c}acdef\\0\end{array}$.

\subsection{} Let $\pi'=\pi\setminus \{\alpha_2\}$. By \cite{J6bis}, we may choose \[S=\left\{-\alpha_3, -\alpha_5, \begin{array}{c}01210\\1\end{array}, \begin{array}{c}11110\\1\end{array}, \begin{array}{c}01111\\1\end{array}\right\}\] and \[T=\left\{11111, 11110, 01100, \begin{array}{c}12221\\1\end{array}, \begin{array}{c}11211\\1\end{array}, \begin{array}{c}12321\\2\end{array}\right\}.\]

\subsection{}\label{E6-35} Let $\pi'=\pi\setminus \{\alpha_3\}$. Then $\pi'$ has two connected components $\pi_1'=\{\alpha_1\}$ of type $A_1$ and $\pi_2'=\{\alpha_2, \alpha_4, \alpha_5, \alpha_6\}$ of type $A_4$. The involution $i$ is defined on $\pi$ by $$i(\alpha_1)=\alpha_1,\, i(\alpha_2)=\alpha_6,\, i(\alpha_3)=\alpha_3,\, i(\alpha_4)=\alpha_5.$$
The $\langle ij\rangle$ - orbits in $\pi$ are the $\{\alpha_1,\, \alpha_2,\, \alpha_6\},\, \{\alpha_3,\, \alpha_4, \,\alpha_5\}$. In particular, $\ind\,\trp=2$.

\subsubsection{} Recall the set of strongly orthogonal roots $\beta_\pi$ and Table I. For $S^\pm$ we choose $S^+=\beta_\pi^0=\{\beta_1, \beta_2,\, \beta_3\}$ and $S^-=-\beta_{\pi_2'}=\{-\beta''_1:=-(\alpha_2+\alpha_4+\alpha_5+\alpha_6), -\beta''_2:=-(\alpha_4+\alpha_5)\}$. For $T^\pm$ we choose $T^+=\beta_\pi\cap \pi=\{\beta_4=\alpha_4\}$ and $T^-=-\beta_{\pi'_1}=\{-\alpha_1\}$. One has $|T|=\ind\, \trp$, thus condition (4) of Lemma \ref{regulargeneral} holds.

For every $\gamma\in S=S^+\sqcup S^-$, we choose the Heisenberg set of centre $\gamma$, $\Gamma_\gamma=H_\gamma$. Note that $S^+\sqcup T^+=\beta_\pi$ and $S^-\sqcup T^-=-\beta_{\pi'}$. Then by \ref{STcascade}, conditions (1) and (2) of Lemma \ref{regulargeneral} are satisfied. Finally, an easy computation of the determinant $\det\, (\alpha_i^\vee(s_j))$, with $1\le i\le 6,\, i\ne 3$ and $S=\{s_j\,|\, 1\le j\le 5\}$ settles condition (3).

\subsection{} Finally, let $\pi'=\pi\setminus \{\alpha_4\}$. Then $\pi'$ has three connected components, $\pi_1'=\{\alpha_1, \alpha_3\},\, \pi'_2=\{\alpha_5, \alpha_6\}$ both of type $A_2$ and $\pi_3'=\{\alpha_2\}$ of type $A_1$. The involution $i$ is given by $$i(\alpha_1)=\alpha_3,\, i(\alpha_2)=\alpha_2,\, i(\alpha_5)=\alpha_6,\, i(\alpha_4)=\alpha_4.$$
Hence there are four $\langle ij\rangle$ - orbits in $\pi$, namely $\{\alpha_1,\, \alpha_5\},\, \{\alpha_3,\, \alpha_6\},\, \{\alpha_2\},\, \{\alpha_4\}$.

\subsubsection{} We choose for $S^\pm$ the sets $S^+=\{\beta_1,\, 11110,\, 01111\},\, S^-=-\beta_{\pi'}^0=\{-\beta_1':=-(\alpha_1+\alpha_3),\, -\beta_1'':=-(\alpha_5+\alpha_6)\}$ and for $T^\pm$ the sets $T^+=\{\beta_2, \,\beta_4,\, \alpha_6\}, T^-=\{-\alpha_2\}$. In particular $|T|=4=|E|$.\\

\noindent {\bf Remark.} Recall Section \ref{cascade}. Note that $\Delta_{\beta_1}=\Delta\setminus (H_{\beta_1}\sqcup -H_{\beta_1})$ is a root system of type $A_5$. The choice of $\{11110,\, 01111,\, -11000,\, -00011\}$ coincides with the choice of $S$ for $\pi$ of type $A_5$ and 
$\pi'$ is equal to $\pi$ without its third root as constructed in \cite{J5}.

\subsubsection{} For $\gamma\in \beta_\pi\sqcup (-\beta_{\pi'})$ we set $\Gamma_\gamma=H_\gamma$. For the rest of the elements of $S$, namely $11110$ and $01111$ we choose \[\Gamma_{11110}=\{10000,\, 11000, \,11100,\, 01110, \, 00110, \, 00010,\, 11110\}\] and
\[\Gamma_{01111}=\{01000,\, 01100,\, 00111,\, 00011,\, 01111\}.\] They are clearly Heisenberg sets of centres $11110$ and $01111$ respectively. Conditions (1) and (2) of Lemma \ref{regulargeneral} 
are easily checked by hand.

\subsubsection{} Set $s_1=\beta_1,\, s_2=11110,\, s_3=01111,\, s_4=-\beta_1',\, s_5=-\beta_1''$ and $h_1=\alpha_1^\vee,\, h_2=\alpha_2^\vee,\, h_3=\alpha_3^\vee,\, h_4=\alpha_5^\vee,\, h_5=\alpha_6^\vee$. Then $$\det\, (h_i(s_j))=\left|\begin{array}{ccccc} 0&1&-1&-1&0\\
1&-1&-1&0&0\\
0&0&1&-1&0\\
0&1&0&0&-1\\
0&-1&1&0&-1
\end{array}\right|=3\ne 0.$$
We conclude that $S|_{\lieh_E}$ is a basis for $\lieh^*$. Hence all conditions of Lemma~\ref{regulargeneral} are satisfied.

\section{Type $E_7$}\label{E7}
Let $\g$ be of type $E_7$ and let $\pi'=\pi\setminus \{\alpha_s\}$, for $1\le s\le 7$, $\trp$ the associated parabolic. By 
Theorem~\ref{borel}
$Y(\trp)$ is a polynomial algebra for $s=2,\, 5$. By \cite{PPY}, $Y(\trp)$ is polynomial for $s=1$ - an adapted pair has been constructed in this case in \cite{J6bis}. In the sections below, we construct an adapted pair for all $\trp$, for $s=2,\, 5$ and we give the construction of \cite{J6bis} for $s=1$.

Again we are using the notations of \cite[Planches I-IX]{BOU} and we denote by $\begin{array}{c}acdefg\\b\end{array}$ the root $a\alpha_1+b\alpha_2+c\alpha_3+d\alpha_4+e\alpha_5+f\alpha_6+g\alpha_7$. As before, if $b=0$, we simply write $acdefg$ instead of $\begin{array}{c}acdefg\\0\end{array}$.

Recall that $j=\id$.

\subsection{} Take $\pi'=\pi\setminus \{\alpha_1\}$. By \cite{J6bis} one may choose \[S=\left\{-\alpha_2, -\alpha_4, -\alpha_6, \begin{array}{c}112111\\  1   \end{array}, \begin{array}{c}112210\\ 1    \end{array}, \begin{array}{c}122110\\ 1    \end{array}\right\}\] and \[T=\left\{\begin{array}{c}012221\\ 1    \end{array}, \begin{array}{c}012210\\ 1    \end{array}, \begin{array}{c}123321\\ 2    \end{array}, \begin{array}{c}012100\\ 1    \end{array}, \begin{array}{c}123221\\ 1    \end{array}, \begin{array}{c}122211\\ 1    \end{array}, \begin{array}{c} 234321\\ 2    \end{array}\right\}.\]

\subsection{} Let $\pi'=\pi\setminus \{\alpha_2\}$ - then $\pi'$ is of type $A_6$. One has that $i$ is the involution defined by

 $$i(\alpha_1)=\alpha_7,\, i(\alpha_3)=\alpha_6,\, i(\alpha_4)=\alpha_5,\,i(\alpha_2)=\alpha_2.$$

 The $\langle ij\rangle$ - orbits are $\{\alpha_1,\, \alpha_7\},\, \{\alpha_2\},\, \{\alpha_3,\, \alpha_6\},\, \{\alpha_4, \,\alpha_5\}$, hence $\ind\, \trp=4$.

\subsubsection{} For $S^\pm$ we choose $S^+=\beta_\pi^0=\{\beta_1,\, \beta_2,\, \beta_3\}$ and $S^-=-\beta_{\pi'}=\{-\beta_1',\, -\beta_2',\, -\beta_3'\}$ and for $T^\pm$ we choose $T=T^+=\beta_\pi\cap \pi=\{\alpha_7,\, \alpha_2,\, \alpha_3,\, \alpha_5\}$. One has $|T|=4=\ind\, \trp$. Then $S^+\sqcup T^+=\beta_\pi$ and $S^-=-\beta_{\pi'}$.

\subsubsection{} For $\gamma\in S$ we take $\Gamma_\gamma=H_\gamma$; by \ref{STcascade} condition (1) and (2) of Lemma \ref{regulargeneral} hold for these choices. Finally, $S|_{\lieh_E}$ is a basis for $\lieh_E^*$, by an easy calculation of the determinant.

\subsection{} Now let $\pi'=\pi\setminus \{\alpha_5\}$. Then $\pi'$ has two connected components, $\pi_1'=\{\alpha_1,\, \alpha_2,\, \alpha_3,\, \alpha_4\}$ of type $A_4$ and $\pi'_2=\{\alpha_6,\, \alpha_7\}$ of type $A_2$. We have that $i$ is the involution given by $$i(\alpha_1)=\alpha_2,\, i(\alpha_3)=\alpha_4,\, i(\alpha_5)=\alpha_5,\, i(\alpha_6)=\alpha_7.$$ Hence there are four $\langle ij\rangle$ - orbits in $\pi$, namely the $\{\alpha_1,\, \alpha_2\},\, \{\alpha_3,\, \alpha_4\},\, \{\alpha_5\}, \{\alpha_6,\, \alpha_7\}$.

\subsubsection{} One has that $\beta_{\pi_1'}=\{\beta_1'=\alpha_1+\alpha_2+\alpha_3+\alpha_4, \beta_2'=\alpha_3+\alpha_4\}$ and $\beta_{\pi_2'}=\{\beta_1''=\alpha_6+\alpha_7\}$. For $S^\pm$ we choose $S^+=\beta_\pi^0=\{\beta_1,\, \beta_2,\, \beta_3\}$ and $S^-=(-\beta_{\pi_1'})\sqcup (-\beta_{\pi_2'})=-\beta_{\pi'}$. For $T^\pm$ we choose $T=T^+=\beta_\pi\cap \pi=\{\alpha_2,\, \alpha_3,\, \alpha_5,\, \alpha_7\}$. We have $|T|=4=|E|$. Again $S^+\sqcup T^+=\beta_\pi$ and $S^-=-\beta_{\pi'}$. By taking $\Gamma_\gamma=H_\gamma$ for all $\gamma\in S$,  conditions (1)--(4) of Lemma \ref{regulargeneral} follow as in the previous case.


\section{Type $E_8$}\label{E8}

Let $\g$ be of type $E_8$ and let $\pi'=\pi\setminus \{\alpha_s\}$, for $1\le s\le 8$. By 
Theorem~\ref{borel}, $Y(\trp)$ is a polynomial algebra for $s=3$. By \cite{Y}, $Y(\trp)$ is not polynomial for $s=8$. Below, we construct an adapted pair for $\trp$, for $s=3$.

\subsection{} Let $\pi'=\pi\setminus \{\alpha_3\}$ - it has two connected components, $\pi'_1=\{\alpha_1\}$ of type $A_1$ and $\pi'_2=\{\alpha_2,\, \alpha_4,\, \alpha_5,\, \alpha_6,\, \alpha_7, \,\alpha_8\}$ of type $A_6$. Then $j=\id$ and
$i$ is the involution that fixes $\alpha_1$ and $\alpha_3$ and interchanges $\alpha_2$ and $\alpha_8$, $\alpha_4$ and $\alpha_7$, $\alpha_5$ and $\alpha_6$. Thus there are five $\langle ij\rangle$-orbits in $\pi$, namely $\{\alpha_1\},\, \{\alpha_3\}, \,\{\alpha_2,\, \alpha_8\},\, \{\alpha_4,\, \alpha_7\},\, \{\alpha_5,\, \alpha_6\}$. One has that $\ind\, \trp=5$.

\subsubsection{} For $S^\pm$ we choose $S^+=\beta_\pi^0=\{\beta_1,\, \beta_2,\, \beta_3,\, \beta_4\}$ and $S^-=-\beta_{\pi'_2}=\{-\beta_1',\, -\beta_2', \,-\beta_3'\}$. For $T^\pm$ we choose $T^+=\beta_\pi\cap \pi=\{\alpha_2,\, \alpha_3,\, \alpha_5,\, \alpha_7\}$ and $T^-=-\beta_{\pi_1'}=\{-\alpha_1\}$. In particular, $|T|=5=\ind\, \trp$. We take $\Gamma_\gamma=H_\gamma$ for all $\gamma\in S$ and we conclude as in type $E_7$.



\section{Type $F_4$}\label{F}
Let $\lieg$ be of type $F_4$, $\pi' =\pi \setminus \{\alpha_s\}$, where $1\leq s\leq 4$, $\trp$ the corresponding truncated parabolic. Recall that we are using the labeling of \cite[Planche VIII]{BOU}.
In particular, we denote by $abcd$ the root $a\alpha_1+b\alpha_2+c\alpha_3+d\alpha_4$. Recall also the set of strongly orthogonal roots of $\lieg$ listed in Table I, $\beta_{\pi} = \{\beta_1=2342, \, \beta_2=0122,\, \beta_3=0120,\,
\beta_4=0100\}$.

By Theorem~\ref{borel}, 
$Y(\trp)$ is polynomial for $s=2, \,3, \,4$. This is still true for $s=1$ by \cite{PPY}; moreover, in this case, an adapted pair has been constructed in \cite{J6bis}. For the sake of completeness, we give the construction of \cite{J6bis} in Section \ref{F4-1} below. We construct an adapted pair for $s=2,\, 4$. However, we show with the help of a computer that adapted pairs do not exist for $s=3$.

Indeed, in this case $Y(\trp^-)$ is a polynomial algebra on three generators of degrees $3,\, 4$ and $10$. If $(h,\, y=\sum\limits_{s\in S}x_s)$ is an adapted pair for $\trp$, then $|S|=3,\, h(s)=-1$ for all $s\in S$ and $2,\, 3,\, 9$ are the eigenvalues of $h$ on a complement of $(\ad\, \trp^-)\,y$. For any $h=\lambda\alpha_1^\vee+\mu\alpha_2^\vee+\nu\alpha_4^\vee\in \h_E$, its eigenvalues on the $28$ root vectors of $\trp$ are functions on $\lambda,\, \mu,\, \nu$ and may be computed by hand.  We computed all $h$ having $\{-1,\, -1, \,-1,\, 2,\, 3, \,9\}$ among their eigenvalues on $\trp$. Then by setting $S$ to be the set of the three eigenvectors of eigenvalue $-1$, we showed that $y$ fails to be regular, and so adapted pairs do not exist.

Recall that $j=\id$.

\subsection{}\label{F4-1} Let $\pi'=\{\alpha_2,\, \alpha_3,\, \alpha_4\}$. Then by \cite{J6bis}, we may choose $S=\{-0010,\, 1121,\, 1220\}$ and $T=\{2342,\, 1222,\,  0122,\, 0111\}$.
\subsection{}
Let $\pi' =\{\alpha_1, \,\alpha_3,\, \alpha_4\}$,
so that the Levi factor of $\liep_{\pi', E}$ is the product
$\liel_{\pi'} = \liesl_2\times \liesl_3$. Equivalently, $\pi'$ consists of two connected components, namely $\pi_1'=\{\alpha_1\}$ of type $A_1$ and $\pi_2'=\{\alpha_3,\, \alpha_4\}$ of type $A_2$. We then have three
$\langle ij\rangle=\langle i\rangle$ - orbits in 
$\pi$, namely $\Gamma_1 =
\{\alpha_1\}, \, \Gamma_2 = \{\alpha_2\},\, \Gamma_3 = \{\alpha_3,\,
\alpha_4\}$. Hence $\ind\, \trp=3$.

One has that $\beta_{\pi'} = \beta_{\pi'_1}\sqcup \beta_{\pi'_2}$, with $\beta_{\pi'_1}=\{\beta':=\alpha_1\},\, \beta_{\pi_2'}=\{\beta'':=\alpha_3+\alpha_4\}$.

\subsubsection{} We choose for $S$ the set $S =
\{\underbrace{\beta_1,\, \beta_2-\alpha_4=\beta_3+\alpha_4}_{S^+}, \underbrace{-\beta''}_{S^-}\}$ and for
$T$ the set $T = \{\underbrace{\beta_2,\, \beta_4}_{T^+}, \underbrace{-\beta'}_{T^-}\}$. Notice that $|T|=3=\ind\,\trp$, hence condition (4) of Lemma \ref{regulargeneral} holds.

\subsubsection{} Set $\Gamma_{\beta_1} = H_{\beta_1},\,\Gamma_{\beta_2-\alpha_4} = H_{\beta_2}^0\bigsqcup H_{\beta_3}$
and $\Gamma_{-\beta''} = -H_{\beta''}$. These are Heisenberg sets (for $\Gamma_{\beta_2-\alpha_4}$ see Lemma \ref{T(C)} - one may as well check by hand since the cardinality of this set is only $7$). Their union is disjoint since the $H_{\beta},\, \beta\in\beta_{\pi},\, -H_{\beta},\, \beta\in
\beta_{\pi'}$ are disjoint. Finally, by comparison of $\Gamma^\pm:=\bigsqcup\limits_{\gamma\in S^\pm}\Gamma_\gamma$ and \ref{heisenbergconditions} (2), the complement of $\Gamma^\pm$ in $\Delta^+$ (resp. $\Delta^-_{\pi'}$) is $T^\pm$ and then condition (1) of Lemma \ref{regulargeneral} is satisfied. On the other hand, condition (2) follows as in Lemma \ref{C2}.\\

\noindent {\bf Remark.} Notice that $\Delta\setminus \left(H_{\beta_1}\sqcup -H_{\beta_1}\right)$ (which is equal to $\Delta_{\beta_1}$ with the notation of \ref{cascade}) is a root system of type $C_3$ (with highest root $0122$). The choice of the roots $\beta_2-\alpha_4$ and $-(\alpha_3+\alpha_4)$ coincides with the choice of $S$ for $\pi$ of type $C_3$ and 
$\pi'$ is equal to $\pi$ without the last root.


\subsubsection{} It remains to verify condition (3) of Lemma \ref{regulargeneral}. \\

\begin{lemma}\label{typeF2support} One has that $S|_{\lieh _E}$ is a basis for
$h^*_E$.
\end{lemma}
\begin{proof} Let $s_1=\beta_1, s_2=\beta_2-\alpha_4, s_3=-\beta''$ be the elements of $S$ and
take $\{h_1=\alpha_1^\vee, \,h_2=\alpha_3^\vee,\, h_3=\alpha_4^\vee\}$ a basis of $\lieh_E$. It is sufficient to
check that the determinant of $(h_i(s_j))_{i, j}$ is non zero. Indeed, one checks that
$$\det\, (h_i(s_j))=\left|\begin{array}{ccc} 1&-1&0\\
0&1&-1\\
0&0&-1
\end{array}\right|=-1\ne 0.$$
\end{proof}

\subsection{} Let $\pi' = \{\alpha_1,\, \alpha_2,\,
\alpha_3\}$, so that the Levi factor of $\liep_{\pi',\, E}$ is of
type $B_3$. Then $i=\id$, hence the $\langle ij\rangle$ - orbits of the Dynkin diagram are the
singletons, $\{\alpha_i\},\, 1\le i\le 4$ and so $\ind\,\trp=4$.

\subsubsection{} We choose $S =\{\underbrace{\beta_1,\, \beta_2 -\alpha_4 = \beta_3+\alpha_4}_{S^+},\, \underbrace{-\beta_1'}_{S^-}\}$ and
 $T = \{\underbrace{\beta_2,\, \beta_4}_{T^+},\,
\underbrace{-\beta'_{1'},\, -\beta_2'}_{T^-}\}$. Notice that $|T|=4=\ind\, \trp$, hence condition (4) of Lemma \ref{regulargeneral} holds.

\subsubsection{} We set $\Gamma_{\beta_1} =H_{\beta_1},\, \Gamma_{-\beta_1'}=-H_{\beta_1'}$ and $\Gamma_{\beta_2-\alpha_4} =
H_{\beta_2}^0\sqcup H_{\beta_3}$. It is clear that $\Gamma_{\beta_1}\sqcup \Gamma_{\beta_2-\alpha_4}$ (resp. $-\Gamma_{\beta_1'}$) complements $T^+$ (resp. $T^-$) in $\Delta^+$ (resp. $\Delta^-_{\pi'})$. Moreover, as before, these are Heisenberg sets, thus condition (1)  of Lemma \ref{regulargeneral} is satisfied. Condition (2) follows as in \ref{C2}. Finally, condition (3) follows by an easy computation as in \ref{typeF2support}.

\section{Type $G_2$}
Let $\lieg$ be of type $G_2$ with $\pi = \{\alpha, \,\beta\}$ a
set of simple roots and $\alpha$ short. Then $Y(\trp)$ is polynomial for both maximal truncated parabolics of $\g$. 

\subsection{} Let $\pi'=\{\alpha\}$. An adapted pair has been constructed in \cite{J6bis}. The choices for $S$ and $T$ are $\{\alpha+\beta\}$ and $\{3\alpha+2\beta,\, 3\alpha+\beta\}$ respectively.

\subsection{} Let $\pi'=\{\beta\}$. Take $S=\{\beta_1=3\alpha+2\beta\}$ and $T=\{\alpha, \,-\beta\}$ and $\Gamma_{\beta_1}=H_{\beta_1}$. It is immediate to verify conditions (1)-(4) of Lemma \ref{regulargeneral}.

\end{document}